\documentclass[11pt]{amsart}
\usepackage[english]{babel}
\usepackage{amssymb,amsthm,amsmath,amstext}
\usepackage{mathrsfs}
\usepackage{bm}
\usepackage[utf8]{inputenc}
\usepackage{amscd}
\usepackage{tikz}
\usepackage{version}
\usepackage{comment}
\usepackage{version}
\usepackage{marginnote}

\usepackage[color,matrix,arrow]{xy}
\xyoption{all}
\usepackage{color}
\usepackage{tikz-cd}

\DeclareMathOperator{\Coh}{\textbf{Coh}}

\DeclareMathOperator{\NS}{NS}
\DeclareMathOperator{\Stab}{Stab}

\DeclareMathOperator{\Mov}{Mov}

\DeclareMathOperator{\Bir}{Bir}

\makeatletter
\ifnum\@ptsize=0  \fi \ifnum\@ptsize=2  \fi 

\title{More birational involutions}
\author{Pietro Beri, Laurent Manivel}
\date{}

\address{Institut Elie Cartan de Lorraine, Université de Lorraine et CNRS, 54000 Nancy, France}
\email{pietro.beri@univ-lorraine.fr}

\address{Institut de Math\'ematiques de Toulouse ; UMR 5219, Universit\'e de Toulouse \& CNRS, F-31062 Toulouse Cedex 9, France}
\email{manivel@math.cnrs.fr}

\begin{document}

\theoremstyle{plain}
\newtheorem{theorem}{Theorem}[section]
\newtheorem*{conjecture*}{Conjecture}
\newtheorem{prop}[theorem]{Proposition}
\newtheorem{lemma}[theorem]{Lemma}
\newtheorem{coro}[theorem]{Corollary}
\theoremstyle{definition}
\newtheorem{definition}[theorem]{Definition}
\newtheorem{remark}[theorem]{Remark}
\newtheorem*{ack}{Acknowledgements}

\def\AA{{\mathbb{A}}}
\def\FF{{\mathbb{F}}}
\def\CC{{\mathbb{C}}}
\def\RR{{\mathbb{R}}}
\def\OO{{\mathbb{O}}}
\def\HH{{\mathbb{H}}}
\def\PP{{\mathbb{P}}}
\def\QQ{{\mathbb{Q}}}
\def\ZZ{{\mathbb{Z}}}
\def\SS{{\mathbb{S}}}
\def\cA{{\mathcal{A}}}
\def\cB{{\mathcal{B}}}
\def\cC{{\mathcal{C}}}
\def\cD{{\mathcal{D}}}
\def\cE{{\mathcal{E}}}
\def\cF{{\mathcal{F}}}
\def\cG{{\mathcal{G}}}
\def\cH{{\mathcal{H}}}
\def\cI{{\mathcal{I}}}
\def\cJ{{\mathcal{J}}}
\def\cK{{\mathcal{K}}}
\def\cL{{\mathcal{L}}}
\def\cM{{\mathcal{M}}}
\def\cO{{\mathcal{O}}}
\def\cP{{\mathcal{P}}}
\def\cQ{{\mathcal{Q}}}
\def\cR{{\mathcal{R}}}
\def\cS{{\mathcal{S}}}
\def\cT{{\mathcal{T}}}
\def\cU{{\mathcal{U}}}
\def\cV{{\mathcal{V}}}
\def\cZ{{\mathcal{Z}}}
\def\ra{{\rightarrow}}
\def\lra{{\longrightarrow}}
\def\ft{{\mathfrak t}}\def\fs{{\mathfrak s}}
\def\fg{{\mathfrak g}}\def\fp{{\mathfrak p}}\def\fb{{\mathfrak b}}
\def\fso{\mathfrak{so}}
\def\fsp{\mathfrak{sp}}
\def\fsl{\mathfrak{sl}}\def\fgl{\mathfrak{gl}}

\let\Iff\iff
\def\iff{\Leftrightarrow}

\begin{abstract} For $S$ a very general polarized $K3$ surface of degree $8n-6$, we describe in geometrical terms a birational involution of the Hilbert scheme $S^{[n]}$ of $n$ points on the surface, whose existence was established in \cite{beri} from lattice theoretical considerations. In \cite{invol1} we  studied this involution for $n=3$ with the help of the exceptional Lie group $G_2$, since the Mukai model of $S$ is embedded in its projectivized Lie algebra. Here we use different, more general arguments to show that some important features of the birational involution persist for $n\ge 4$. In particular we describe the indeterminacy locus of the involution in terms of a Mori contraction, and deduce that it is birational to a $\PP^2$-fibration on a moduli space of sheaves $\Sigma$ on $S$, that also admits a degree two nef and big line bundle and an induced birational involution.
\end{abstract}

\maketitle

\section{Introduction}

In 1983, Beauville described explicitly a birational involution of the punctual Hilbert scheme $S^{[n]}$, for $S$ a K3 surface whose Néron-Severi group $\NS(S)=\mathbb{Z}\cH$ is generated by the class of an ample divisor $\cH$ of self-intersection $\cH^2=2n$ \cite[section 6]{beauville-remarks}.
Later on, under similar hypothesis,
O'Grady described an involution on $S^{[2]}$ when $\cH^2=10$ \cite[section 4.3]{ogrady}, and recently 
the authors of the present paper described an involution on $S^{[3]}$ when $\cH^2=18$ \cite{invol1}. 
Some other cases have been studied, but we focus on the latter ones because they
belong to a sequence. 
Indeed the following result is  proved in \cite{beri}: for any $n\geq 2$ and $k\ge 1$,  and any 
K3 surface $S$ whose Néron-Severi group is generated by a class of self-intersection $2t$,
with $t=(n-1)k^2+1$, the group of birational automorphisms of $S^{[n]}$
is generated by a non-trivial birational involution.

After Beauville's work from  1983, only a limited number of geometric descriptions for automorphisms of Hilbert schemes of points on very general algebraic K3 surfaces have been
found (see \cite{debarre} for a survey). Here we describe a new infinite family of such involutions.
Beauville's involution, as it has been called, appears when  $k=1$ and so $t=n$; here we take care of the cases where  $k=2$ and $t=4n-3$. Lastly
 these constructions were extended in \cite{faenzi-menet} to moduli spaces of sheaves on K3 surfaces admitting special spherical objects. 

The starting point of our study is the existence of a rank two Mukai bundle $\cU_2^\vee$ on $S$. It embeds the surface into the Grassmannian $G(2,V_{2n+1})$ of planes in a complex vector space of dimension $2n+1$, and identifies with the restriction of the dual tautological bundle. In particular, the zero locus of any section of $\cU_2^\vee$ is of the form $G(2,V_{2n})\cap S$ for some hyperplane $V_{2n}\subset V_{2n+1}$. Via the embedding of $S$ into the Grassmannian, any finite, length $n$ subscheme $Z$ of $S$ generates a linear space in $V_{2n+1}$. A simple observation allows to describe the birational involution $\varphi$ of $S^{[n]}$ in very concrete terms.

\begin{prop}\label{intro-involution}
For a generic $Z\in S^{[n]}$, there exists a unique (up to nonzero scalar) section $s$ of $\cU_2^\vee$ 
vanishing along $Z$. The zero-locus $Z(s)$ is finite of length $2n$, and 
$$Z(s)=Z\cup \varphi(Z).$$
Equivalently, a generic $Z\in S^{[n]}$ generates a hyperplane $V_{2n}\subset V_{2n+1}$, and $\varphi(Z)$ 
is the residual scheme of $Z$ in $G(2,V_{2n})\cap S$.
\end{prop}

This construction is clearly similar to Beauville's one, for which the image of a generic length $n$ subscheme $Z$ of $S\subset \PP^{n+1}$ is the residual scheme in the intersection of $S$ with the hyperplane spanned by $Z$ in $\PP^{n+1}$.
\medskip

Using Bayer and Macri's approach to wall-crossing \cite{bayermacri}, we prove the following result, for a very general polarized K3 surface $S$ of degree $8n-6$ (see Theorem \ref{thm I=J} and Corollary \ref{structure I}):

\begin{theorem}\label{intro-indeterminacy} The indeterminacy locus $I$ of $\varphi$ coincides with the locus $J$ of schemes $Z\in S^{[n]}$ that do not generate 
a hyperplane in $V_{2n+1}$. Moreover, $I=J$ is birational to a 
$\PP^2$-fibration over a moduli space $\Sigma$ of sheaves on $S$. 
\end{theorem}

This moduli space of sheaves $\Sigma$ also admits a nef and big line bundle $\cL$ whose Beauville-Bogomolov degree is equal to $2$, and inherits a birational involution $\varphi_\Sigma$ that belongs to the class of involutions uncovered in \cite{faenzi-menet} (Proposition \ref{descr varphiSigma}). 

We prove several refinements of the previous statement; in particular there exists a natural stratification of $I=J$ by the dimension of the span in $V_{2n+1}$, and this stratification if well-behaved. In particular, each stratum is birational to a fibration over an auxiliary moduli space, and the dimension of the fiber in which a subscheme lies only depends on the dimension of its linear  span in $V_{2n+1}$
(Theorem \ref{descr I}). The construction of this stratification follows from a good understanding of the possible (derived) Jordan-H\"older filtrations of the ideal sheaves of finite subschemes in $S^{[n]}$, considered as a moduli space of semistable objects in the derived category of $S$.

We also describe $\varphi$ in terms of the "other" boundary of the nef cone of $S^{[n]}$, generated by $\cH_n-2\delta$ (the obvious  
boundary is generated by $\cH_n$, the line bundle induced by the polarization $\cH$ of $S$, and the corresponding contraction is the Hilbert-Chow morphism; as usual $2\delta$ denotes the class of the divisor of non-reduced schemes). 
Up to constant, the class of  $\cH_n-2\delta$ is the only one that is fixed by the action of $\varphi$ in cohomology. The associated linear system plays an important r\^ole in the geometry of the involution.

\begin{theorem}\label{intro-basepointfree}
The divisor $\cH_n-2\delta$ is base-point free when it is  nef and big. The morphism it defines is generically finite of degree two onto its image, and $\varphi$ is the corresponding covering
involution. 
\end{theorem}

This is Theorem \ref{bpf for nC=1} and Corollary \ref{phi is covering inv}. The condition that 
 $\cH_n-2\delta$  is  nef and big is expected to hold in full generality; we have checked it by computer for $n\le 200$ (Proposition \ref{nC=1}).

 \medskip
These results were proved in \cite{invol1} for $n=3$ by making heavy use of the geometry of $G_2$. Here we provide different, more general arguments; although we obtain less precise results, they are less specific and more relevant from the prespective of the general theory of sheaves on $K3$ surfaces and stability conditions in derived categories, whose full force we don't refrain from using. 

The moduli space $\Sigma$ over which the indeterminacy locus fibers birationally, admits a natural map 
$$\Sigma\lra \overline{\Sigma}\subset \PP^{\frac{(n-2)(n+1)}{2}},$$ 
whose image $\overline{\Sigma}$ we call the Pl\"ucker variety.
For $n=4$, $\Sigma$ is a double EPW sextic and we recover a familiar set-up. The cases where $n\ge 5$ would certainly deserve further investigations, as they 
might lead to  new projective models of hyperK\"ahler manifolds, that we plan to investigate. 

\medskip
The structure of the paper is the following. In the next section we discuss the geometry of zero-loci of global sections of the Mukai bundle on $S$, including the geometric description of the birational involution  $\varphi$ which is the content of Proposition \ref{intro-involution}. Then we start in section 3  our study of the linear system $\cH_n-2\delta$, proving that the associated morphism is generically finite of degree two, the corresponding covering involution being $\varphi$. Section 4 discusses the walls and chambers decomposition of the nef and movable cones of $S^{[n]}$. The proof of the first part of Theorem \ref{intro-basepointfree} is the object of section 5. The much longer and more technical section 6 focuses on the wall-crossing interpretation of $\varphi$ and the corresponding flopping contraction, which is amenable to a quite precise description in terms of moduli spaces of objects in the derived category of $S$; we deduce Theorem \ref{intro-indeterminacy} and its refinements. In particular the base of the flopping contraction is birational to the moduli space $\Sigma$, endowed with a nef and big class $\cL$ of Beauville-Bogomolov degree $2$. It follows from works of Oguiso that for certain values of $n$, $\Sigma$ has an  infinite 
group of birational transformations; we show that the special one induced by $\varphi$ is 
precisely one of the birational involutions recently described in \cite{faenzi-menet}. We conclude the paper by showing that the morphism defined by $\cL$ maps $\Sigma$ to the Pl\"ucker variety 
$\overline{\Sigma}$, whose geometry would certainly deserve to be investigated further. A brief Appendix contains the computer program that was used to prove Proposition \ref{nC=1}.
 
\medskip\noindent {\it Acknowledgements.} We warmly thank D. Faenzi, G. Kapustka, E. Macri,  G. Mongardi, F. Giovenzana, K. O'Grady for their comments and hints. 

P. Beri is  supported by the ANR project “Positivity on K-trivial varieties", grant ANR-23-CE40-0026, and 
L. Manivel by the ANR project "FanoHK", grant
ANR-20-CE40-0023.

\medskip 

\section{The involution}

\subsection{Existence of the birational involution}

In \cite{beri} were completely characterized the pairs $(n,t)$ for which $S^{[n]}$ admits 
a non-trivial birational automorphism, $S$ being a K3 surface whose N\'eron-Severi group is generated by a class $\cH$ of self-intersection $2t$. This characterization is stated in terms of the Pell equation 
$$X^2-t(n-1)Y^2=1.$$
For $t=4n-3$, the minimal positive solution of this equation is $(z,w)=(8n-7,4)$. Applying 
 \cite[Theorem 1.1]{beri}, we deduce: 

 \begin{prop}\label{action}
$S^{[n]}$ admits a unique non-trivial birational automorphism $\varphi$. This is 
a non-symplectic involution,  acting 
by $\varphi_*= 
-R_{\cH_n-2\delta}$ on $H^2(S^{[n]},\ZZ)$.
 \end{prop}

 Here as usual we denoted by $\cH_n$ the divisor on $S^{[n]}$ induced by $\cH$, and by $2\delta$ the class of the divisor $E$ parametrizing non-reduced schemes.  
 Moreover, for $D$ non-isotropic, $R_D$ denotes the reflection with 
 respect to $D$, which sends  $D$ to $-D$ and fixes the orthogonal hyperplane. 
 By \cite[Proposition 2.1]{beri}, $\varphi_*= -R_D$ for $D=b\cH_n-a\delta$, $(a,b)$ being the minimal solution of the Pell equation $(n-1)X^2-tY^2=-1$, which in our case is $(2,1)$.
 Actually this  is stated  only for the action on the Neron-Severi group, but the action on the transcendental lattice is by minus the identity,  see \cite[Proof of Theorem 1.1 (i)]{beri}.

\subsection{The Mukai bundle}
The constructions of \cite{ogrady} and \cite{invol1}, corresponding respectively to $n=2$ and $3$, rely on a very explicit
projective model of the K3 surface, namely a \textit{Mukai model}. But such a description 
is not available  for $t\geq 19$, and no explicit geometric model is in fact 
expected for $t>61$ \cite{ghs}. Instead we use the following result \cite[Theorem 3]{mukai}. 

\begin{theorem}
Let $S$ be a K3 surface whose N\'eron-Severi group is generated by a class $\cH$ of self-intersection $2t$.
For any two integers $r,s\ge 2$ such that $rs=t+1$, there exists a unique stable vector bundle 
$\cE$ of rank $r$ on $S$ with first Chern class $\cH$ and such that $\chi(\cE)=r+s$. Moreover,
\begin{enumerate}
    \item $\cE$ is globally generated and has no higher cohomology,
    \item the natural map $\lambda:\bigwedge^r H^0(S,\cE)\to H^0(S,\bigwedge^r\cE)$ is surjective.
\end{enumerate}
\end{theorem}

We call $\cE$ the Mukai bundle of rank $r$. Its Mukai vector is $v=(r,\cH,s)$, so that $v^2=-2$ 
and $\cE$ is therefore an {\it exceptional} vector bundle,
also called a {\it rigid} bundle if it is simple (see \cite[Proposition 3.2]{mukai-tata}),
or also a {\it spherical} vector bundle \cite{huybrechts-chowspherical}.
The existence of $\cE$ was also proved by Kuleshov \cite{kuleshov}, following ideas of Mukai 
who treated the 
case where $v^2\ge 0$ in \cite{mukai-tata}. (Actually Kuleshov only proved the existence of a semistable 
vector bundle with a given exceptional Mukai vector. See \cite{bkm2} for an update.) 
Uniqueness had previously been observed by Mukai \cite[Corollary 3.5]{mukai-tata}. That $\cE$ is generated by global sections also follows from 
\cite[Proposition 8.10]{coskun-neuer-yoshioka}, that it has no higher cohomology  from 
\cite[Theorem 8.3]{coskun-neuer-yoshioka}. Assertion (2) is not considered in  \cite{kuleshov}, and stated without proof in \cite{mukai}. We will prove a stronger statement 
in Proposition \ref{grass} (for rank two, but the proof extends to arbitrary rank).

\medskip
As for line bundles, we denote by $\phi_{\cE}$ the morphism defined by the sections of $\cE$. Since $\cH=\det(\cE)$, we get a commutative diagram
\[\begin{tikzcd}
S  \ar[d,"\phi_{\cH}",hook]\ar[r,"\phi_{\cE}", hook] & G(r,H^0(S,\cE)^\vee) \ar[d,"p",hook] \\
 \PP(H^0(S,\cH)^\vee) \ar[r,"\lambda^\vee", hook] & \PP(\wedge^r H^0(S,\cE)^\vee)
\end{tikzcd}\]
where $p$ is the Pl\"ucker embedding. 

\medskip 
In our case where $t=4n-3$, we let $(r,s)=(2,2n-1)$. Then $H^0(S,\cE)$ has dimension $2n+1$ and we will denote its dual by $V_{2n+1}$. We get 
\[ S\subset G(2,V_{2n+1})\cap L\subset \PP(\wedge^2 V_{2n+1}) \]
with $L=\PP(H^0(S,\cH)^\vee)\cong \PP^{t+1}$. Note that $L$ is highly non-generic, since its 
codimension is quadratic in $n$, while it meets non-trivially the Grassmannian $G(2,V_{2n+1})$
whose dimension is linear in $n$. 

\begin{definition}
    We denote the Mukai bundle $\cE$ by $\cU_2^\vee$; it is the restriction to $S$ of the dual tautological bundle on the Grassmannian $G(2,V_{2n+1}).$
\end{definition}

Under the embedding $S\subset G(2,V_{2n+1})$, intersections of $S$ with Schubert cycles
of type  $G(2,V_{2n})$ are zero-loci of  sections of $\cU_2^\vee$ (note that  
$V_{2n}$ is a hyperplane inside
$V_{2n+1}=H^0(S,\cU_2^\vee)^\vee$, so identifies with a line in $H^0(S,\cU_2^\vee)$). The following statement is close to Lemma 3.1 and Proposition 3.2 in \cite{voisin-K3}. 

\begin{prop}\label{prop: deg small grassmannian}
Let $s$ be any non-zero section of $\cU_2^\vee$. Then its zero-locus $Z(s)$ is a finite scheme
of length $2n$. Moreover, $s$ is the only section of $\cU_2^\vee$ (up to scalar) vanishing along $Z(s)$.
\end{prop}

\begin{proof}
Since $\cU_2^\vee$ is generated by global sections, the zero-locus of a general section is a set of simple points, and the number of these points,  since 
 $ch_2(\cU_2^\vee)=\frac{1}{2}c_1(\cU_2^\vee)^2-c_2(\cU_2^\vee)$,  is 
 $$\int_S c_2(\cU_2^\vee)=\frac{1}{2}\int_S c_1(\cU_2^\vee)^2 -\Big(\chi(\cU_2^\vee)-2\chi(\cO_S)\Big)=
\frac{1}{2}(8n-6) - (2n-3) = 2n.$$
More generally, this will be the length of $Z(s)$ as soon as it is a finite scheme,
which is always the case since if $s$ happened to vanish along a curve $C$, it would embed $\cO_S(C)$ as a subsheaf of $\cU_2^\vee$, and this would contradict the stability of $\cU_2^\vee$. 

Now since $Z(s)$ is finite,  
the associated Koszul complex is exact. Twisting it by $\cU_2^\vee$,
we get the exact sequence 
$$0\lra \cU_2\lra End(\cU_2)\lra I_{Z(s)}\otimes \cU_2^\vee\lra 0.$$
We know that $H^1(S,\cU_2)=H^1(S,\cU_2^\vee)^\vee=0$, and that $H^0(S,End(\cU_2))=\CC$ since $\cU_2^\vee$ is stable. So 
$H^0(S,I_{Z(s)}\otimes \cU_2^\vee)=\CC$, which is  our last claim. 
\end{proof}

\begin{remark}
It was proved in \cite{huybrechts-chowspherical} and \cite{voisin-K3} that the second Chern class of $\cU_2^\vee$ has to be a multiple of the canonical zero-cycle of the surface $S$.
\end{remark}

\begin{remark}
A consequence of Proposition \ref{prop: deg small grassmannian} 
is that  sending a section to its 
zero-locus defines an injective map
$$\gamma : \PP H^0(S,\cU_2^\vee)=\PP^{2n}\lra S^{[2n]},$$
whose image is of course a Lagrangian cycle. Using the Koszul complex it is easy to 
show that the differential of $\gamma$ at $[s]$ is given by the evaluation map 
$$H^0(S,\cU_2^\vee)/\CC s \lra \cU^\vee_{2|Z(s)}.$$
The fact that $H^0(S,I_{Z(s)}\otimes \cU_2^\vee)=\CC$ implies that this differential is everywhere injective, 
so that $\gamma$ is an embedding. 

This reflects the positivity of $\cU_2^\vee$ in a certain way. 
According to \cite[Proposition 2]{mop}, $\cU_2^\vee$ is $(k-1)$-very ample for $2n+1\ge 3k$, which means that for any finite subscheme $Z$ of $S$ of length $k$, the evaluation map 
$$H^0(S,\cU_2^\vee)\lra H^0(Z,\cU^\vee_{2|Z})$$
is surjective. 
\end{remark}

The next observation will be useful later on.
For any nonzero section $s$ of $\cU_2^\vee$, the finite set $Z(s)$ has a linear 
span $\langle Z(s)\rangle$ of dimension $2n-2$. Indeed the Koszul complex, twisted by $\cH$, yields an exact sequence 
$$0\lra \CC s\lra H^0(S,\cU_2^\vee)\lra H^0(S,I_{Z(s)}\otimes \cH)\lra 0.$$
More precisely, this exact sequence yields the following statement. 

\begin{prop}\label{covers} For any nonzero section $s$ of $\cU_2^\vee$, the orthogonal of $\langle Z(s)\rangle$ in $H^0(S,\cH)$ is the image of $s\wedge H^0(S,\cU_2^\vee)$ by the wedge product map. 

For any $z\in L=\PP H^0(S,\cH)^\vee$, there exists a non zero section $s$ of $\cU_2^\vee$ such that $z\in \langle Z(s)\rangle$. The space of such sections is the kernel of the skew-symmetric 
form $\psi_z$ on $H^0(S,\cU_2^\vee)$ defined by $z$. 
\end{prop}

\proof The first assertion follows from the Koszul complex. 
The second assertion is a direct consequence of the fact that 
$H^0(S,\cU_2^\vee)$ has odd dimension. Indeed, the skew-symmetric form $\psi_z$ 
on $H^0(S,\cU_2^\vee)$ defined by $z$ must have a non-trivial kernel; moreover,
 $s$ being in the kernel means that $z$ is orthogonal 
to $s\wedge H^0(S,\cU_2^\vee)$, hence that $z\in \langle Z(s)\rangle$
according to the first claim. \qed 

\begin{coro}\label{unique Z(s)}
A generic point $z\in L\simeq \PP^{4n-2}$ lies on a unique linear subspace 
$\langle Z(s)\rangle\simeq\PP^{2n-2}$,
$[s]\in \PP H^0(S,\cU_2^\vee)\simeq \PP^{2n}$. 
\end{coro}

This yields a remarkable distribution of disjoint linear spaces covering an open subset of $\PP^{4n-2}$.
The geometry of this distribution will play a crucial r\^ole in the proof of Proposition \ref{card fiber}.

\medskip Now let us consider pencils of sections. If $s_1, s_2$ are two sections of $\cU_2^\vee$, let us denote by $\pi(s_1\wedge s_2)$ the corresponding section of $\cH$. First observe that if 
$\pi(s_1\wedge s_2)=0$, then $s_1$ and $s_2$ must be proportional. Indeed their values must be proportional at any point of $S$, hence $s_1/s_2$ defines a rational function on $S$. But if this function has p\^oles along a curve $C$, then 
$s_2$ vanishes along this curve and defines a nonzero section of 
$\cU_2^\vee(-C)$, which contradicts the stability of $\cU_2^\vee$. We thus get a regular morphism from $G(2,H^0(S,\cU_2^\vee))$ to $ \PP H^0(S,\cH)$.

\begin{prop}\label{grass}
The morphism  $\eta: G(2,H^0(S,\cU_2^\vee))\lra \PP H^0(S,\cH)$ is surjective and finite.
\end{prop}

\proof Observe that the source and target have the same dimension $4n-2=t+1$. 
 The morphism is dominant by \cite[Corollary 4.3]{huybrechts-chowspherical},
 hence surjective, and therefore finite since the Grassmannian has cyclic Picard group. \qed

\medskip Since this projection $\eta$ is linear with respect to the Pl\"ucker embedding, its degree is just the degree of the Grassmannian $G(2,2n+1)$, which is equal to the Catalan number $\frac{1}{2n}\binom{4n-2}{2n-1}$. 
Denote by $\cR\subset G(2,H^0(S,\cU_2^\vee))$ the ramification divisor. 

\begin{lemma}\label{ramification}
    The pencil $P=\langle s_1, s_2\rangle$ defines a point of $\cR$ if and only if $\langle Z(s_1)\rangle$ meets $\langle Z(s_2)\rangle$. 
\end{lemma}

\proof The differential of $\eta$ at $P$ is induced by the 
morphism that sends $(t_1,t_2)\in H^0(S,\cU_2^\vee)$ to $\pi(s_1\wedge t_1)-\pi(s_2\wedge t_2)$. We deduce that this differential fails to be injective exactly when $\pi(s_1\wedge H^0(\cU_2^\vee))$ and 
$\pi(s_2\wedge H^0(\cU_2^\vee))$ have a bigger intersection than $\pi(P)$. Taking orthogonals, this exactly means that 
$\langle Z(s_1)\rangle $ and $\langle Z(s_2)\rangle $ meet non trivially. \qed

\begin{remark}\label{D}
By the same argument,  $P=\langle s_1, s_2\rangle\in \cR$ is a simple ramification point exactly when $Z(s_1)$ and $Z(s_2)$ meet at a single point. Let $\cD\subset L$ be the locus of points $z$ where the skew-symmetric form $\psi_z$, see Proposition \ref{covers}, has corank at least three; so $\cD$ is a Pfaffian locus, defined by Pfaffian equations of degree $n$. Let $\cI$ be the set of pairs $(z, P=\langle s_1, s_2\rangle)$ such that $z\in Z(s_1)\cap Z(s_2)$. In the diagram 
$$\xymatrix{\cI\ar[r]\ar[d] & \cR\subset G(2,H^0(S,\cU_2^\vee)) \\
\cD &}$$
the projection $\cI\ra\cR$ is thus birational over the component $\cI_0$ of $\cI$ dominating 
$\cR$, whose dimension is therefore $4n-3$. On the other hand, $\cD$ has everywhere dimension at least $4n-5$, with fibers $\PP^2$ over points $z$ where $\psi_z$
has corank exactly three, and of bigger dimension when the corank is bigger. 
So necessarily the projection $\cI_0\ra\cD$ is generically a $\PP^2$-fibration over a component $\cD_0$ of $\cD$, on the generic point of which the corank of $\psi$ is 
exactly three.  
\end{remark}

\subsection{A monodromy result}
The following result will be useful later on.  We denote by $\Gamma$ the monodromy group 
of the finite subsets of $S$ given as zero loci of general sections of $\cU_2^\vee$. If we fix such a section $s_0$ and let $Z_0=Z(s_0)$, $\Gamma$ is identified with a  
subgroup of the permutation group $S(Z_0)\simeq S_{2n}$, the latter identification being 
well-defined up to conjugation. 

\begin{prop}\label{monodromy}
$\Gamma$ is the full permutation group.
\end{prop}

\proof We follow the classical approach of \cite{harris-galois}. Let $R\subset \PP H^0(S,\cU_2^\vee)\times S$ denote the set of pairs $([s],p)$ such that $s(p)=0$. Since $\cU_2^\vee$ is generated by global sections, each fiber of the projection of $R$ to $S$ is a codimension two linear subspace of $\PP H^0(S,\cU_2^\vee)$. This implies that $R$ is irreducible. Using the same argument as 
in \cite[p.698]{harris-galois}, we conclude  that $\Gamma$ acts transitively on $Z_0$. 

According to \cite[Proposition 2]{mop}, 
$\cU_2^\vee$ is $k$-very ample when $3k\le 2n-2$, hence $1$-very ample as soon as $n\ge 3$. Fix a point $p$ in $Z_0$ and denote by $R_p$ the 
corresponding fiber of the projection of $R$ to $S$, which is just the linear system of 
sections of $\cU_2^\vee$ that vanish at $p$. Let $R'_p\subset R_p\times (S-\{p\})$ denote the set of  pairs $([s],q)$ such that $q\ne p$ and $s(q)=s(p)=0$. Since $\cU_2^\vee$ is $1$-very ample, all the 
fibers of the projection of $R_p'$ to $S-\{p\}$ are codimension four linear spaces. We deduce that $R_p'$ is irreducible, and therefore, by the same argument as before, 
the stabilizer of $p$ in $\Gamma$ has to act transitively on $Z_0-\{p\}$. We conclude 
that the action of $\Gamma$ on $Z_0$ is doubly transitive. 

Recall that the discriminant subvariety of  $\PP H^0(S,\cU_2^\vee)$ is the locus $\Delta_{\cU_2^\vee}$ parametrizing sections whose zero-locus is not reduced. Since $\cU_2^\vee$ is $2$-very ample for $n\ge 4$, the same dimension count as above ensures that  $\Delta_{\cU_2^\vee}$ is an irreducible hypersurface, 
whose generic point is given by a section of $\cU_2^\vee$ whose zero-locus is made of $2n-2$ simple points and a non-reduced scheme of length two. Using \cite[II.3, Lemma p.698]{harris-galois}, we deduce 
that $\Gamma$ contains at least one simple transposition. But then, being doubly-transitive 
it contains all the simple transpositions, hence it must be the full symmetric group. \qed

\subsection{Geometric description of the involution}
 Let us finally describe the birational involution $\varphi$ geometrically. 
Consider $n$ points $p_1,\ldots,p_n$ on $S$, in general position. Let $P_1,\ldots , P_n\subset V_{2n+1}$ denote the corresponding planes, and let $V_{2n}$ denote their span. Since $p_1,\ldots,p_n$ are in general position, $V_{2n}$ is a general hyperplane in $V_{2n+1}$. In
other words, the points $p_1,\ldots,p_n$ impose independent conditions on sections of $\cU_2^\vee$, and 
there is a unique section $s$ (up to scalar) such that $s(p_1)=\cdots = s(p_n)=0.$

By Proposition \ref{prop: deg small grassmannian}, the zero-locus $Z(s)=G(2,V_{2n})\cap S$ consists of $2n$ simple 
points, $n$ of which being $p_1,\ldots,p_n$. Denoting the remaining points by $q_1,\ldots,q_n$, we can then define explicitely the birational involution $\varphi$ of  $S^{[n]}$ by 
\[ \varphi:  p_1+\cdots+p_n\mapsto q_1+\cdots+q_n. \]
Proposition \ref{intro-involution} clearly holds true.

\begin{remark}\label{Beauville-like}
     If instead one takes $t=n$ and $(r,s)=(1,t+1)$, the same construction as above leads to Beauville's construction.
\end{remark}

\medskip

\section{The linear system}

Recall that the second cohomology group  
\begin{equation}\label{H2dec}
H^2(S^{[n]},\ZZ)\simeq H^2(S,\ZZ)\oplus \ZZ\delta, 
\end{equation}
where $2\delta$ is the class of the divisor $E$ of non-reduced schemes;
the embedding of $H^2(S,\ZZ)$ is given by
$\xi\mapsto \xi_n:=HC^* (\xi^{(n)})$, where $HC:S^{[n]}\to S^{(n)}$ is the 
Hilbert-Chow morphism and $\xi^{(n)}$ is induced by $\xi$ on $S^{(n)}$. 

\smallskip
\begin{lemma}
The class $\cH_n-2\delta$ has Beauville-Bogomolov degree 2.
\end{lemma}

\proof This follows immediately from \cite[3.2.1]{debarre}. \qed 

\medskip
Proposition \ref{action} implies that the divisor  $\cH_n-2\delta$ is fixed by the 
induced action of $\varphi$ on the N\'eron-Severi group of $S^{[n]}$. This motivates the study of its linear system, which is the object of this section.

\subsection{Polynomials vanishing on the $(n-2)$-th secant}
 From the definition of the line bundles $\cH_n$ and $2\delta$, we get  a clear geometric interpretation by reasoning as in \cite[Proposition 39]{invol1}, namely
\[|\cH_n-2\delta|\simeq \PP(H^0(S^{(n)},\cH^{\otimes n}\otimes\cI_{\Delta})\simeq I_n(Sec^{n-2}(S)),\]
the system of degree $n$ hypersurfaces in $L$ that contain the $(n-2)$-th secant variety of $S$ (we denoted by $\Delta$ the big diagonal in $S^{(n)}$). In other words, this is the space of hypersurfaces that are maximally singular on $S$, in the sense that their polynomial equations have all their derivatives of order $n-2$ vanishing on $S$.

\begin{prop}
The linear system $|\cH_n-2\delta|$ has dimension $\frac{n(n+3)}{2}$.
\end{prop}

\begin{proof}
Since $\cH_n-2\delta$ belongs to the movable cone, 
there exists a smooth birational model $f:M\dashrightarrow S^{[n]}$ 
for which $f^*(\cH_n-2\delta)$ is nef and big. Then it has no higher cohomology, and we can use the Hirzebruch-Riemann-Roch \cite[(5)]{debarre} to compute
$$h^0(M,f^*(\cH_n-2\delta))=\chi(M,f^*(\cH_n-2\delta))=\frac{(n+1)(n+2)}{2}.$$ 
Finally, since the indeterminacy locus of $f$ has codimension 
at least $2$,  $H^0(M,f^*(\cH_n-2\delta))\simeq H^0(S^{[n]},\cH_n-2\delta)$.
\end{proof}

The morphism $\phi_{\cH_n-2\delta}: S^{[n]}\dashrightarrow |\cH_n-2\delta|^\vee$ sends a generic (hence reduced, and not in the base locus of the linear system) scheme $Z=p_1+\cdots +p_n$ 
to the hyperplane in $I_n(Sec^{n-2}(S))$ parametrizing hypersurfaces
that contain the linear span $\langle p_1,\ldots,p_n \rangle $. 
That this is a codimension one linear condition follows from the fact that the $(n-2)$-th secant variety of $n$ points in general position is just the simplex generated by these points. This 
simplex is the union of $n$ hyperplanes in $\langle p_1,\ldots,p_n \rangle $, hence a 
hypersurface of degree $n$. So a degree $n$ hypersurface $(P=0)$ containing the simplex will 
contain its linear span, if and only if it contains any other given point in 
$\langle p_1,\ldots,p_n \rangle $. 
If we denote in the same way $P$ and its polarization, this linear condition can simply be stated as $$P(p_1,\ldots ,p_n)=0.$$

\begin{remark}
For an alternative description, observe that 
$$H^0(S^{(n)},\cH^{\otimes n}\otimes\cI_{\Delta})=Ker (S^nH^0(S,\cH)\lra H^0(S,\cH^{\otimes 2})\otimes S^{n-2}H^0(S,\cH)).$$
From this point of view, $p_1+\cdots +p_n$ defines a hyperplane in $S^nH^0(S,\cH)$, hence 
(in general) in this kernel, simply by evaluating sections of $\cH$ at 
the $n$ points $p_1,\ldots , p_n$.
\end{remark}

\subsection{The Pfaffian subsystem and its base locus}\label{sec: pfaff}\label{sec pfaffians}
There is a distinguished linear subsystem in $|\cH_n-2\delta|,$  
given by
$$\begin{array}{rcl}
 p_{2n+1}:   V_{2n+1} & \to & I_n(Sec^{n-2}(S))\\
    v        & \mapsto & P_v=v\wedge \_ \wedge \cdots \wedge \_.
\end{array}$$
Here we have fixed a volume form, that is, a generator of $\wedge^{2n+1}V_{2n+1}$, 
allowing to identify $P_v(\omega)=v\wedge \omega\wedge\cdots\wedge\omega$ ($n$ times)
with a scalar. 

\smallskip
Inside $\PP(\wedge^2V_{2n+1})$, the same formula defines a linear system of degree $n$ 
polynomials $P_v$ whose base locus is the variety of skew-symmetric tensors which are not of 
maximal rank $2n$. In particular, since $S$ parametrizes rank two tensors, $P_v$ certainly 
vanishes on $Sec^{n-2}(S)$. 

\begin{definition}\label{def pfaff}
    The image of $p_{2n+1}$ will be also denoted $V_{2n+1}$. We call the associated linear subsystem the \emph{Pfaffian subsystem}.
\end{definition}

Observe that a reduced scheme $Z=p_1+\cdots +p_n$ belongs to the base locus  of this
subsystem exactly when $p_1,\ldots ,p_n$ do not impose independent conditions 
on $H^0(S,\cU_2^\vee)$. We will denote this base locus by $J$. More generally, a finite scheme 
$Z$ defines a point of $J$ when the evaluation map $H^0(S,\cU_2^\vee)\lra H^0(Z,\cU^\vee_{2|Z})$ is not 
surjective. 

We always have $h^2(\cI_Z\otimes \cU_2^\vee)=0$ since $Z$ is zero-dimensional and $\cU_2^\vee$ has no higher cohomology. Moreover an easy computation yields $\chi(\cI_Z\otimes \cU_2^\vee)=1$. Hence the following Brill-Noether type description of $J$. 

\begin{lemma}\label{J-BrillNoether}
    A finite scheme $Z\in S^{[n]}$ belongs to $J$ if and only if 
   $$h^1(\cI_Z\otimes \cU_2^\vee)=h^0(\cI_Z\otimes \cU_2^\vee)-1>0.$$
\end{lemma}

\begin{remark}
 If    $h^0(\cI_Z\otimes \cU_2^\vee)>1$, there is a pencil $\langle s_1, s_2\rangle$ of sections of $\cU_2^\vee$ vanishing at $Z$. Then $s_1\wedge s_2$ is a section of $\cH=\det (\cU_2^\vee)$, hence defines a hyperplane section which is singular along $Z$.  
\end{remark}

Recall that $ H^0(Z,\cU^\vee_{2|Z}) $ is the fiber at $Z$ of the tautological vector bundle 
of rank $2n$ on $S^{[n]}$ usually denoted $(\cU_2^\vee)^{[n]}$. We can then see $J$ globally as
the degeneracy locus of the vector bundle morphism 
\begin{equation}\label{eval}
H^0(S,\cU_2^\vee)\otimes \mathcal{O}_{S^{[n]}}\lra (\cU_2^\vee)^{[n]}.
\end{equation}

\begin{prop}\label{surJ}
    $J$ is a non-empty subscheme of $S^{[n]}$, of codimension at most two 
    at any of its points.
\end{prop}

\proof The second claim follows from the usual properties of degeneracy loci. 
For the first claim, we observe that if $J$ was empty, the evaluation morphism 
(\ref{eval}) would be surjective, and its kernel would be a line bundle. As a 
consequence the Segre classes  $s_k((\cU_2^\vee)^{[n]})$ would vanish for any $k>1$. 
But according to \cite[Theorem 3 (i)]{mop},
$$\int_{S^{[n]}} s_{2n}((\cU_2^\vee)^{[n]})=(-3)^n\binom{2n-2}{n-2}\ne 0,$$
a contradiction. \qed

\begin{lemma}
$p_{2n+1}$ is injective.
\end{lemma}

\proof
Suppose $P_v$ is identically zero on $L$. Since $S\subset L$ is nondegenerate, 
this is equivalent to 
the condition that the polarization $P_v(p_1,\ldots ,p_n)=0$ for any $p_1,\ldots ,p_n\in S$. 
But each $p_i$ is identified with the codimension two subspace of sections of $\cU_2^\vee$ vanishing at 
$p_i$. If we chose a basis $a_i,b_i$ of the fiber of $\cU_2^\vee$ at $p_i$, and consider the dual
basis  $a_i^\vee,b_i^\vee$, this means that $p_i$ is identified with  $[\alpha_i\wedge \beta_i]$ where 
 $\alpha_i=a_i^\vee\circ ev_{p_i}$ and $\beta_i=b_i^\vee\circ ev_{p_i}$.

Now, $P_v(p_1,\ldots ,p_n)=v\wedge \alpha_1\wedge \beta_1\wedge\cdots \wedge \alpha_n\wedge \beta_n$. So suppose $s$ is a generic section, and that $p_1,\ldots , p_n$ are points in $Z(s)$. By the previous Proposition they impose independent conditions on sections of $\cU_2^\vee$  (recall that $H^0(S,\cU_2^\vee)=V_{2n+1}^\vee$), 
which means that $\alpha_1, \beta_1, \ldots, \alpha_n,\beta_n$ are independent. Thus 
$P_v(p_1,\ldots ,p_n)=0$ means that $v$ belongs to their linear span. In particular $v$ 
has to vanish on $s$, since $\alpha_1, \beta_1, \ldots, \alpha_n,\beta_n$ do by construction. 
But $s$ is generic, so necessarily $v=0$. \qed

\begin{remark} Alternatively, one can observe that 
$$\det((\cU_2^\vee)^{[n]})=\cH_n-2\delta$$
(see \cite[Lemma 1.5]{wandel}), and that $p_{2n+1}$ is induced by the  morphism
$$\wedge^{2n}H^0(S,\cU_2^\vee)\simeq \wedge^{2n}H^0(S^{[n]},(\cU_2^\vee)^{[n]})\lra  \hspace*{5cm}$$
$$  \hspace*{5cm} \lra H^0(S^{[n]}, \wedge^{2n}(\cU_2^\vee)^{[n]})\simeq H^0(S^{[n]},\cH_n-2\delta).$$
\end{remark}

\subsection{The generic degree}
Our  linear systems define a commuting diagram
\begin{equation}\label{subsystem map}
\xymatrix{ &  S^{[n]}\ar@{..>}[dr]^{\phi_{V_{2n+1}}}\ar@{..>}[dl]_{\phi_{\cH_n-2\delta}} & \\  \PP(I_{n}(Sec^{n-2}(S))^\vee)\ar@{..>}[rr]_{p_{V_{2n+1}^\vee}} & &  \PP(V_{2n+1}^\vee) }
\end{equation}

\smallskip\noindent 
where $p_{V_{2n+1}^\vee}$ is a linear projection. We will prove later on that 
$\phi_{\cH_n-2\delta}$ is  regular under the hypothesis that $\cH_n-2\delta$ is nef and big, which we expect to hold true in full generality (see Theorem \ref{bpf for nC=1} and Proposition \ref{nC=1}). 

\begin{lemma}\label{factors via quotient, def arg}
$\phi_{\cH_n-2\delta}$ factors through the quotient by $\varphi$.
\end{lemma}

\begin{proof}
This follows from \cite[Corollary 4.7]{debarre}, since $\cH_n-2\delta$ becomes ample on a general small deformation of $(S^{[n]},\cH_n-2\delta)$.
\end{proof}

The rational map $\phi_{V_{2n+1}}$
to $\PP(V_{2n+1}^\vee)=\PP H^0(S,\cU_2^\vee)$ sends a scheme $Z\notin J$ to the unique (up to scalar)
section $s$ of $\cU_2^\vee$ that vanishes on $Z$.  Since in general $Z(s)$ consists in $2n$ 
simple points, $\phi_{V_{2n+1}}$ is generically $\binom{2n}{n}$-to-$1$.

\begin{coro}
$\phi_{\cH_n-2\delta}$ is generically finite of degree $d\ge 2$ dividing  $\binom{2n}{n}$.
\end{coro}

\begin{remark}
For $n=3$, we gave in \cite{invol1} a geometric proof of the fact that 
$\phi_{\cH_n-2\delta}$ factorizes through $\varphi$, based on the fact that the 
$2n$ points of $Z(s)$ are not in general linear position. Indeed, the twisted Koszul 
complex 
$$ 0\lra \cO_S\lra \cU_2^\vee\lra I_{Z(s)}\otimes \cH\lra 0$$
implies that $h^0(S, I_{Z(s)}\otimes \cH)=2n$, so that the linear span $\langle Z(s)\rangle\simeq \PP^{2n-2}$ has one dimension less than expected.
\end{remark}

In \cite{invol1} we also proved for $n=3$ that the degree of the generically finite morphism 
$\phi_{\cH_n-2\delta}$ is $d=2$. In fact this is always the case.

\begin{prop} \label{card fiber}
    $d=2$. 
\end{prop}

\proof Recall the result of Proposition \ref{monodromy}, which is that the monodromy group of the zero-locus of the generic section of $\cU_2^\vee$ is the full symmetric group. Denote the $2n$ points of this zero-locus by  $p_1,\ldots , p_n, q_1, \ldots , q_n$ as above. We know that $p_1+\cdots +p_n$ and $q_1+\cdots +q_n$ are in the same fiber
$F$ of $\phi_{\cH_n-2\delta}$, and that any other point of $F$ must be of the form $p_I+q_J$
for some $I,J\subset \{1,\ldots ,n\}$ of respective sizes $i$ and $j=n-i$. If such a 
point does exist with $i,j>0$, then by  monodromy  all the points $p_{I'}
+q_{J'}$ with $I'$ and $J'$ of the same sizes $i$ and $j$ must also belong to $F$. 

In other words, there exist $i,j>0$, with $i+j=n$, such that each time we consider a 
generic point $p_1+\cdots +p_n$ of $S^{[n]}$, all the other points $p_I+q_J$
with $I,J$ of respective sizes $i$ and $j$, belong to the same fiber. So let us apply 
this to $p'_1+\cdots +p'_n=p_I+q_J$ and $q'_1+\cdots +q'_n=p_{I^c}+q_{J^c}$. We get 
in the same fiber all the points $p'_{I'}+q'_{J'}$, where $I'$ is obtained by choosing
$i$ points among the $p_I$'s and $q_J$'s, and $J'$ is obtained by choosing $j$ points among the $p_{I^c}$'s and $q_{J^c}$'s. Suppose the respective numbers of these points are $i-a, a, j-b, b$ respectively; note that the size of $I^c$ is $j$, so the only constraints are that 
$0\le a,b\le i,j$. Then we get in the fiber a point of the form $p_A+q_B$ with $B$
of size $a+b$, and we also get $p_{A^c}+q_{B^c}$ where $A^c$ has size $a+b$. 

We conclude that 
each time we have a 
point $p_I+q_J$ in the fiber $F$, with $i,j>0$, we also get all the $p_K+q_L$ 
with $K$ of size at most 
$2\min (i,j)$. Iterating this process if necessary, we will exhaust all the possible values 
of $i$ and $j$, which means that in fact, all the   $p_K+q_L$'s are contained in the 
fiber $F$. Let us exclude this last possibility.

\begin{lemma} $d\ne \binom{2n}{n}$. \end{lemma}

We proceed by contradiction. This would imply that there is a hyperplane in $I_n(Sec^{n-2}(S))$ consisting of $n$-ics vanishing 
on all the linear spaces spanned by any $n$ points among $p_1,\ldots , p_n, q_1,\ldots ,q_n$. Recall moreover that the linear span $\langle Z(s)\rangle$ of these points has codimension $2n$ in $L$, 
hence dimension $2n-2$. Since the points are in general position, such an $n$-ic has to vanish,
and we conclude that vanishing on $\langle Z(s)\rangle$ is a codimension one condition on $I_n(Sec^{n-2}(S))$.

Let us fix a general nonzero section $s_0$ and a general polynomial $P_0$ in $I_n(Sec^{n-2}(S))$, that does not vanish on $\langle Z(s_0)\rangle$. For any $s$ in a neighbourhood of $s_0$, there
is then a linear form $\alpha_s$ on  such that the polynomials that vanish on $\langle Z(s)\rangle$ are those
of the form $P-\alpha_s(P)P_0$. This implies that 
$P(z)/P_0(z)$ is constant on $\langle Z(s)\rangle$.
Let us prove that this constant does in fact not depend on $s$, 
which will imply that $P$ is a multiple of $P_0$ and yield the 
contradiction we are aiming at. 

Consider two generic points $z,z'\in L$; by Corollary \ref{unique Z(s)}, they belong respectively to $\langle Z(s)\rangle$ and $\langle Z(s')\rangle$ for two generic sections of $\cU_2^\vee$, uniquely defined up to scalar. Since belonging to the ramification divisor $\cR$ is a codimension one condition, we can find a section $t$ such that the lines $\langle s,t\rangle$
and $\langle s',t\rangle$ belong to $\cR$. Then by Lemma \ref{ramification},$ \langle Z(s)\rangle$ and $\langle Z(t)\rangle$ meet at some point $x\in\cD$,
and similarly  $\langle Z(s')\rangle$ and $\langle Z(t)\rangle$ meet at some point $x'\in\cD$ (recall Remark \ref{D}). Moreover, since $z$ is generic in $L$, $x$ is generic in $\cD$. 

By definition, $\cD$ is cut-out by the degree $n$ Pfaffian equations, so we can find $P_0\in I_n(Sec^{n-2}(S))$ that does not vanish identically on $\cD$. Since $x$, and symmetrically $x'$, are generic in $\cD$, $P_0(x)$ and $P_0(x')$ do not vanish. 
But then for any $P\in I_n(Sec^{n-2}(S))$, since $P/P_0$ is constant on each linear space $\langle Z(u)\rangle$, it must take the same value at $z$ and at $x$, then at $x$ and at $x'$, then at $x'$ and at $z'$. Since it takes the same value at the generic points $z$ and $z'$, $P/P_0$ must therefore be constant, which is impossible.  
\qed

\begin{coro}\label{phi is covering inv}
$\varphi$ is the covering involution associated to  $\phi_{\mathcal{H}_n-2\delta}$.
\end{coro}

\medskip

\section{Chambers in the movable cone}

\subsection{Walls and chambers}\label{sec: walls and chambers}
Let $\Mov(M)\subset \NS(M)_{\mathbb{R}}$ be the closure of the movable cone of a projective hyperK\"ahler manifold $M$. According to 
\cite[section 5.2]{markmansurvey}, 
\[ \Mov(M)=\overline{\bigcup_{f:M\dashrightarrow M'} f^*\mathcal{A}(M')}, \]
where $M'$ is any hyperK\"ahler birational model of $M$ and $\mathcal{A}(M')$ is its ample cone. We call \textit{chambers} of $\Mov(M)$ the open convex cones of the form $f^*\mathcal{A}(M')$ for some $M'$, and \textit{walls} the non-empty sets of the form
$\overline{f^*\mathcal{A}(M')}\cap \overline{g^*\mathcal{A}(M'')} \setminus \lbrace 0 \rbrace$.
The wall-and-chamber decomposition plays an important role in the study of the automorphisms of $M$, since any birational automorphism of the manifold fixes the movable cone and sends walls to walls, thus chambers to chambers.
In other words, by \cite[Theorem 1.3]{markmansurvey}, chambers in $\Mov(M)$ parametrize pairs $(M',f)$ up to biregular automorphisms.

Since $S^{[n]}$ is projective, 
any birational automorphism of $S^{[n]}$ acting trivially on $\NS(S^{[n]})$ fixes a K\"ahler class, hence lies in $\mathrm{Aut}(S^{[n]})$.
By the argument of \cite[Lemma 2.4]{bcns}, when $S$ has Picard rank one any automorphism acting trivially on $\NS(S^{[n]})$ must be trivial. Thus the movable cone is not fixed pointwise by $\varphi$, which has therefore to 
exchange the two rays generating the movable cone. One is generated by $\cH_n$, since the latter is nef and big and orthogonal (with respect to the Beauville-Bogomolov form) to the exceptional divisor of 
the Hilbert-Chow morphism.

Since according to Proposition \ref{action}, $\varphi$ acts as $-R_{\cH_n-2\delta}$,  the other ray is 
\[
\varphi^*(\cH_n)=(2t-1)\cH_n-4t\delta,
\]
where as usual $t=4n-3$.
By \cite[Lemma 3.6]{beri}, in our setting $t$ is $n$-irregular \cite[Definition 3.1]{beri}, meaning that $\varphi$ is not biregular on any hyperK\"ahler birational model of $S^{[n]}$. Equivalently, the number of chambers is even and the number of walls is odd, the middle one being 
 generated by $\cH_n-2\delta$. Up to scalar, this is the only class fixed by $\varphi^*$ in $\NS(S^{[n]})$.

\begin{definition}
    We denote by $2C_n-1$ 
    the number of walls in the interior of the movable cone, hence by $2C_n$ the number of chambers.
\end{definition}

We call \textit{first chamber} the ample cone of $S^{[n]}$, \textit{second chamber} the chamber whose closure intersects the closure of the first one along a 1-dimensional space and so on. 
The action of $\varphi$ exchanges the $k$-th chamber with the $(2C_n+1-k)$-th. 

\begin{lemma}\label{unique model}
There are exactly $C_n$ distinct hyperK\"ahler birational models of $S^{[n]}$.
If $(M,f)$ is the pair corresponding to the $k$-th chamber, then the pair corresponding to the $(2C_n+1-k)$-th chamber is $(M,f\circ \varphi)$.
\end{lemma}

\begin{proof}
For any birational, non-biregular map $f:S^{[n]}\dashrightarrow M$, with $M$ hyperK\"ahler, the pullback of the ample cone of $M$ is the interior of one chamber of $\Mov(S^{[n]})$. 
If necessary, by the action of $\varphi$  we can always bring $f^*\mathcal{A}(M)$ among the first $C_n$ chambers. 
By \cite[Theorem 1.3]{markmansurvey}, this proves that there are at most $C_n$ birational models and that the models of the $k$-th and $(2C_n+1-k)$-th chamber differ by $\varphi$.

On the other hand,
for a pair $(M,f)$  corresponding to the $k$-th chamber, and  $(M',f')$ corresponding to the $k'$-th,
with $k'\ne k, 2C_n+1-k$, any biregular map 
$\alpha:M\lra M'$ would induce a birational involution $(f')^{-1}\circ \alpha\circ f$ of $S^{[n]}$,
sending the $k$-th chamber to the $k'$-th. So this would neither be $\varphi$ nor the identity, a contradiction since  $\Bir(S^{[n]})=\lbrace 
id_{S^{[n]}},\varphi \rbrace$.
\end{proof}

\begin{prop}\label{prop: ample cone}
The class $\cH_n-2\delta$ is never ample, and is nef and big if and only if $C_n=1$. In this case, the nef cone is 
$$\mathrm{Nef}(S^{[n]})=\langle \mathcal{H}_n, \mathcal{H}_n-2\delta\rangle $$
and the movable cone has two chambers, exchanged by the action of $\varphi$:
\[ \mathrm{Mov}(S^{[n]})=\mathrm{Nef}(S^{[n]})\cup \varphi^* (\mathrm{Nef}(S^{[n]})). \]
\end{prop}

\begin{proof}
The wall generated by $\cH_n-2\delta$ is the middle wall by definition of $n$-irregularity, so it lies in the closure of the ample cone if and only if $C_n=1$. The second part of the claim is clear.
\end{proof}

By Bayer-Macrì's study of the Mori cone of $S^{[n]}$ (as a cone in $H^2(S^{[n]},\RR)$, see \cite[Section 12]{bayermacri}), there is an extremal 
contraction associated to the wall of $\mathrm{Nef}(S^{[n]})$ lying in the interior of the movable cone.
It is a flopping contraction, in which the contracted locus has codimension at least two. 
We call it $c:S^{[n]}\lra N$, with $N$ a normal irreducible projective variety.

\medskip

Suppose now that $C_n=1$.
By Lemma \ref{unique model}, in this case the second chamber $\varphi^*(\mathcal{A}(S^{[n]}))$ is associated to the pair $(S^{[n]},\varphi)$:
the diagram associated to the flop is then
\begin{equation}\label{phi as flop}
\xymatrix{ S^{[n]}\ar@{..>}[rr]^{\varphi}\ar[dr]_{c} & & S^{[n]}\ar[dl]^{d}  \\
   & N &.}
\end{equation}

This provides an alternative way to see $\varphi$, as a flop.
This will be an important step in the description of the indeterminacy locus of the involution, in Section \ref{sec: indeterminacy}.
Note also that, by Proposition \ref{prop: ample cone}, for $C_n=1$ the contraction $c$ is induced by the linear system $|k(\cH_n-2\delta)|$ for $k\gg 0$, see 
the proof of \cite[Theorem 8.1.3]{matsuki}.

\medskip

The last observations indicate that $\varphi$ will be easier to understand when $C_n=1$.
In fact this condition is fullfilled for the first occurrences of $n$, as the following result shows. We expect it to hold in full generality. 

\begin{prop}\label{nC=1}
$C_n=1$ for $n\leq 200$.
\end{prop}

\begin{proof}
This is a computer calculation, see Appendix \ref{appendix}. 
\end{proof}

\medskip

\section{The base locus}
\label{base locus section}

Let $(M,f)$ be the pair associated to a chamber $\mathcal{C}$ in the movable cone of $S^{[n]}$. Since base-point-freeness implies nefness, by Proposition \ref{prop: ample cone} the hypothesis that  $(f^{-1})^*(\cH_n-2\delta)$ has empty base locus 
implies that $\mathcal{C}$ or $\varphi^*\mathcal{C}$ is the $C_n$-th chamber. In particular,
if $C_n>1$, then by Proposition \ref{prop: ample cone} the linear system $|\cH_n-2\delta|$ has non-empty base locus. 
This is why we impose the condition $C_n=1$ in the Theorem below.

\medskip
This section is devoted to the proof of the following

\begin{theorem}\label{bpf for a model}\label{bpf for nC=1}
If $C_n=1$, then $|\cH_n-2\delta|$ is base-point-free.
\end{theorem}

\begin{remark}
Without any hypothesis on the number of chambers, denoting by $(M,g)$ the model corresponding to the $C_n$-th chamber, the proof of Theorem \ref{bpf for a model} can easily be adapted to prove that $|(g^{-1})^*(\cH_n-2\delta)|$  is base-point-free on $M$.
\end{remark}

As an immediate consequence, by \cite[2.1.28]{lazarsfeld} we deduce:

\begin{coro}\label{c stein}
If $C_n=1$, the contraction $c$ is the first factor of the Stein factorization of $\phi_{\cH_n-2\delta}$.
\end{coro}

The first ingredient of the proof of Theorem \ref{bpf for a model} is the fact that base-point-freeness is an open condition in a family of hyperK\"ahler manifolds, see for example \cite[Proposition 2.6]{varesco}.
The second ingredient is the following well known fact.

\begin{lemma}\label{bpf sur Tn}
For any surface $T$ and $n\geq 2$, if $D$ is a base-point-free divisor on $T$, then $D_n$ is base-point-free on $T^{[n]}$.
\end{lemma}

\begin{proof}
Let $\mathcal{S}_n$ be the permutation group of $n$ elements, so that $T^{(n)}=T^n/\mathcal{S}_n$, and let  $p_i$ the $i$-th projection of $T^n$ to $T$.
Since $D$ is base-point-free, the evaluation map $H^0(T,D)\otimes \mathcal{O}_T\lra \mathcal{O}_T(D)$ is surjective, hence there is a $\mathcal{S}_n$-equivariant surjection
$H^0(T^{n},D)^{\otimes n}\otimes \mathcal{O}_T\lra \mathcal{O}_{T^n}(p_1^*D+\cdots+p_n^*D)$. This yields a surjective map  $$H^0(T^{(n)},D^{(n)})\otimes \mathcal{O}_{T^{(n)}}\simeq Sym^n H^0(T^{n},D)\otimes \mathcal{O}_{T^{(n)}} \lra \mathcal{O}_{T^{(n)}}(D^{(n)}).$$
This proves the statement, since $D_n$ is the pullback of $D^{(n)}$ via the Hilbert-Chow morphism.
\end{proof}

Finally, we will use the Torelli Theorem for hyperK\"ahler manifolds. For that, we need to introduce the global period domain for polarized manifolds and some lattice-theoretic tools.

For any lattice $L$, we denote by $\widetilde{O}^+(L)$ the group of orientation-preserving isometries acting trivially on  the discriminant group of $L$.
For $v\in L$, we denote by $\gamma_L(v)$ its divisibility in $L$ and by $v_*$ the class of $\frac{v}{\gamma_L(v)}$ in the discriminant group of $L$.

We fix the lattice $\Xi=U^{\oplus 3}\oplus E_8(-1)^{\oplus 2}\oplus \ZZ \ell$, with $\ell^2=-2(n-1)$.
For $X$ a manifold
of $K3^{[n]}$-type, a marking is an isometry $\psi:H^2(X,\ZZ)\lra \Xi$. We fix once and for all a connected component $\mathcal{M}$ of the moduli space of marked pairs and any manifold of this kind considered in the sequel will be
taken from this component.
We will use the fact that composing a marking with a monodromy operator does not change the connected component in the moduli space of marked pairs, so by \cite[Lemma 9.2]{markmansurvey} for any given $(X,\psi)\in \mathcal{M}$ and $\alpha\in \widetilde{O}^+(\Xi)$, we have $(X,\alpha\circ \psi)\in \mathcal{M}$ again.

\medskip

Let $\lbrace u,v \rbrace$ be a standard basis for a fixed copy of the hyperbolic plane in $\Xi$ and let $M=U^{\oplus 2}\oplus E_8(-1)^{\oplus 2}\oplus \ZZ \ell\oplus \ZZ (u-v)$ be the orthogonal complement of $u+v$ in $\Xi$. The space 
\[
\left\{ [x]\in \PP(M_\mathbb{C}) \,|\, x^2=0, (x,\bar{x})>0 \right\}
\]
has two connected component. We choose one of them and we call it $\Omega$.

From now on we suppose $n\geq 3$, which is harmless since base-point-freeness for $n=2$ is known from works of O'Grady, see \cite{ogrady}.
The global period domain $\mathcal{P}$ for manifolds of $K3^{[n]}$-type with a 2-polarization (a polarization of Beauville-Bogomolov degree 2) and divisibility 1 is the quotient of $\Omega$ by an arithmetic subgroup $\Gamma$ of $O(M)$ in which $\widetilde{O}^+(M)$ has index two \cite[Lemma 3.7, Proposition 3.8]{bbbf}.
This group has the following property: for any marked pair $(X,\psi)\in \mathcal{M}$, provided with a nef and big class $D\in \NS(X)$ such that the marking $\psi$ send $D$ to $u+v$, the pullback of $\Gamma$ via $\psi$ is the group of monodromy operators on $H^2(X,\ZZ)$ fixing $D$.

Since $\cH_n-2\delta$ is nef and big, but not ample, the period points of our pairs $(S^{[n]},\cH_n-2\delta)$ cover a dense subset of an irreducible divisor in the period domain $\mathcal{P}$, and this divisor lies in the complement of the image of the period map of the moduli space of manifolds of $K3^{[n]}$-type with a $2$-polarization of divisibility $1$. This is \cite[Remark 3.7]{beri}, but we describe more precisely the divisor in the next lemma.

\begin{lemma}\label{period of Sn}
    Let $\lbrace u_1,v_1 \rbrace$ be a standard basis of a copy of $U$ in $M$; consider the vector $\kappa=2(n-1)(u-v)+4(n-1)v_1-\ell$. The period points of pairs $(S^{[n]},\cH_n-2\delta)$ are dense in the irreducible divisor
    \[
    \mathcal{D}_\kappa= (\kappa^\perp \cap \Omega)/\Gamma.
    \]
\end{lemma}

\begin{proof}
    By \cite[Theorem 5.2]{beri}, for $(S,\cH)$, $(S',\cH')$ very general, $S^{[n]}$ and $S'^{[n]}$ are birationally equivalent if and only if $S\cong S'$. As a consequence of the Torelli Theorem for hyperK\"ahler manifolds, this implies that the closure of the locus of period points of pairs $(S^{[n]},\cH_n-2\delta)$ is an irreducible divisor in $\mathcal{P}$.
    
    The orthogonal complement of $\cH_n-2\delta$ in $\NS(S^{[n]})$ is generated by $2(n-1)\cH_n-t\delta$.
    Since $\NS(X)=H^2(X,\ZZ)\cap H^{2,0}(X)^\perp$ for any hyperK\"ahler $X$, proving the statement is equivalent to providing a marking $\psi:H^2(S^{[n]},\ZZ)\lra \Xi$ on $S^{[n]}$ such that $(X,\psi)\in \mathcal{M}$, sending $\cH_n-2\delta$ to $u+v$ and $2(n-1)\cH_n-t\delta$ to $\kappa$.
    A marking $H^2(S,\ZZ)\lra U^{\oplus 3}\oplus E_8(-1)^{\oplus 2}$ on $S$ induces a marking $\psi'$ on $S^{[n]}$ via the natural embedding of lattices $H^2(S,\ZZ)\lra H^2(S^{[n]},\ZZ)$, by sending $\delta$ to $\ell$. We can choose the marking on $S$ so that $(S^{[n]},\psi')\in \mathcal{M}$ and, since $H^2(S,\ZZ)$ is unimodular, $\psi'(\cH_n)=u+tv$ by Eichler's Criterion \cite[Proposition 2.15]{song}; then $\psi'(\cH_n-2\delta)=u+tv-2\ell$. 
    
    Both $\cH_n-2\delta$ and $u+v$ have degree 2 and divisibility 1, so the classes $(\psi'(\cH_n-2\delta))_*$ and $(u+v)_*$ are both zero in the discriminant group of $\Xi$.
    Again by Eichler's Criterion, there exists an isometry $\alpha\in \widetilde{O}^+(\Xi)$ sending $\psi'(\cH_n-2\delta)=u+tv-2\delta$ to $u+v$. So we let $\psi=\alpha\circ \psi'$ and get  that $(S^{[n]},\psi)\in \mathcal{M}$. 
    We can describe explicitly $\alpha$ in terms of Eichler transvections (for their definitions and properties, we refer to \cite[Section 3.1]{GHSEichler}). For any $x,y\in \Xi$ with $(x,x)=0$, there is an Eichler transvection 
    \[
    t(x,y):z\in \Xi\mapsto z-(y,z)x+(x,z)y-\frac{1}{2}(y,y)(x,z)x.
    \]
     We denote by $w$ the difference between $u+tv-2\ell$ and $u+v$.
    By \cite[proof of Proposition 3.3]{GHSEichler}, the composition of three Eichler transvections $t(u_1,-v)\circ t(v_1,w)\circ t(u_1,v)$ is an isometry sending $u+tv-2\ell$ to $u+v$.
    By \cite[(5)]{GHSEichler}, 
    \[
    t(u_1,-v)\circ t(v_1,w)\circ t(u_1,v)=
    t(u_1,-v)\circ t(v_1,w)\circ t(u_1,-v)^{-1},
    \]
    thus the composition lies in $\widetilde{O}^+(\Xi)$ since $t(v_1,w)$ does, see \cite[(10)]{GHSEichler}.     
   Hence $\alpha=t(u_1,-v)\circ t(v_1,w)\circ t(u_1,-v)^{-1}$.
    Now a direct computation shows that $\alpha(2(n-1)\cH_n-t\delta)=\kappa$, as required.
\end{proof}

\medskip\noindent {\it Proof of 
Theorem \ref{bpf for a model}}. 
Let $T$ be a very general double cover of $\mathbb{P}^2$ branched along a nodal sextic; its Néron-Severi lattice is $\ZZ D\oplus \ZZ\Gamma = \langle 2\rangle\oplus \langle -2\rangle$ for $D$ the class of the pullback of $\cO_{\PP^2}(1)$.
We can choose a marking $\psi$ on the Hilbert scheme $T^{[n]}$ such that $(T^{[n]},\psi)$ lies in $\mathcal{M}$, and moreover $\psi(D_n)=u+v$ and $\psi(\delta)=\ell$. In addition, since $\Gamma$ has divisibility one and the Néron-Severi lattice is primitive, we can always choose $\psi$ such that $\psi(\Gamma_n)=u_1-v_1$.
Observe that $\kappa'=2(n-1)(u_1-v_1)$ lies in $\psi(\NS(T^{[n]}))\cap (u+v)^\perp$. 
 By Eichler's Criterion, there exists $\alpha$ in $\widetilde{O}^+(M)$ -- hence in $\Gamma$ -- such that $\alpha(\kappa)=\kappa'$, because
$\kappa'_*=-\ell_*=\kappa_*$
in the discriminant group of $M$ and $(\kappa,\kappa)=(\kappa',\kappa')$.
Hence the period point of $(T^{[n]},D_n)$ in $\cP$ lies in $\mathcal{D}_\kappa$.

Since $D$ is base-point-free, the same holds for $D_n$ by Lemma \ref{bpf sur Tn}. 
Let $(\mathcal{X},\mathcal{L})\lra B$ be a family of pairs whose central fiber is $(T^{[n]},D_n)$, with $\mathcal{X}_b$ smooth hyperK\"ahler and $\mathcal{L}_b$ base-point-free for any $b\in B$. By \cite[Proposition 2.6]{varesco}, we can take $B$ connected of maximal dimension.
By Lemma \ref{period of Sn}, 
for a very general $b$ over a codimension one space in $B$ there will be a $(S^{[n]},\cH_n-2\delta)$ with the same period point as $(\mathcal{X}_b,\mathcal{L}_b)$; by the Torelli Theorem for hyperK\"ahler manifolds, this implies that $S^{[n]}$ and $\mathcal{X}_b$ are birational. Since we supposed that $C_n=1$, this implies that $\mathcal{X}_b$ and $S^{[n]}$ are actually isomorphic by Lemma \ref{unique model}. 

Via the isomorphism, $\mathcal{L}_b$ correspond to a nef and big class of Beauville-Bogomolov degree $2$ on $S^{[n]}$, and the only such class is $\cH_n-2\delta$. Indeed, observe that any degree $2$, nef and big class lies in the interior of the movable cone, since $\cH_n$ has stricly bigger degree and the same holds for an integral generator of the second wall of the movable cone, since the action of $\varphi$ exchanges the two walls. Proceeding as in \cite[proof of Theorem 1.1]{beri}, one then checks that any degree $2$, nef and big class in the interior of the movable cone of $S^{[n]}$ induces a birational involution, since minus the reflection in such a class is a Hodge monodromy
of $H^2(S^{[n]},\ZZ)$; then we are done by unicity of $\varphi$.
\qed

\section{The indeterminacy and the contracted locus}\label{sec: indeterminacy}

\subsection{The small contraction}\label{sec: small c}

We consider in greater detail the contraction $c:S^{[n]}\lra N$ associated to the flop $f: S^{[n]}\dashrightarrow M\simeq S^{[n]}$ as in (\ref{phi as flop}). 
For this it will be convenient to consider $S^{[n]}$ as the moduli space of objects in $D^b(S)$ with Mukai vector $(1,0,-(n-1))$, stable with respect to a generic stability condition $\sigma$; indeed, one can choose $\sigma$ such that those objects are exactly the ideal sheaves $\cI_Z$ for $Z\in S^{[n]}$, see for example \cite[Theorem 4.4]{bayer}.
To simplify the statements, from now on we will always work under the hypothesis that $C_n=1$.

\medskip

\begin{definition} For any contraction, we call \textit{base
of the contraction} the image of the exceptional locus; otherwise said, the locus of points with positive dimensional fibers.
\end{definition}

\medskip
As an interesting consequence of Bayer and Macrì's approach to birational transformations of moduli spaces of objects on K3 surfaces, it is possible to construct, for a general point on the base of the contraction $c$, two curves in $S^{[n]}$ which are contracted to that point, one by $c$ and the other by $d$ (recall \eqref{phi as flop}). This means that the bases of the two contractions $c$, $d$ in \eqref{phi as flop} coincide. This was already observed in \cite[Lemma 29]{invol1} for the case $n=3$, and holds in any dimension.

\smallskip
On the full cohomology $H^*(S,\ZZ)$, we consider  the usual Mukai pairing $((r,\mathcal{L},s), (r',\mathcal{L}',s'))=\mathcal{L}\cdot \mathcal{L}'-rs'-r's$.
We write $$H^*_{alg}(S,\ZZ)=H^0(S,\ZZ)\oplus \NS(S)\oplus H^4(S,\ZZ)$$ and  consider a primitive vector 
$v\in H^*_{alg}(S,\ZZ)$  with $v^2=(v,v)>0$.
A moduli space $M$ of objects on $S$ with Mukai vector $v$ comes with a lattice isometry 
\begin{equation}\label{mukai morphism}
\theta: v^\perp \lra  H^2(M,\ZZ).
\end{equation}
It restricts to an isometry between the orthogonal complement to $v$ in $H^*_{alg}(S,\ZZ)$ and $\NS(M)$, called the Mukai morphism; see \cite[Theorem 3.6]{bayermacri} for the statement in full generality. 

\smallskip 
From now on, we consider $v=(1,0,-(n-1))$.
A flop of $S^{[n]}$, considered as above as a  moduli space of objects on $S$, is induced by a wall in the interior of the movable cone, a so-called flopping wall. Any wall is, via the Mukai morphism, the orthogonal complement to some hyperbolic sublattice of $H^*_{alg}(S,\ZZ)$.
This lattice encodes a lot of information about the flop; we call it the lattice associated to the flop.

In \cite{cattaneo}, Cattaneo translated \cite[Theorem 12.1]{bayermacri} in an explicit numerical form for the special case of Hilbert schemes of points: the lattice $\Lambda$ associated to a flop (or a flopping wall) is 
obtained from 
 a positive, integral solution  $(X,Y)$  of a Pell's equation
\begin{equation}\label{eq:pell murs}
    X^2-4t(n-1)Y^2=\alpha^2-4\rho (n-1),
\end{equation}
with  $X\equiv \pm \alpha \pmod{2(n-1)}$, where either
\begin{equation}\label{eq alpha et rho}
    \begin{cases}
        \rho=-1 \text{ and } \alpha \in \lbrace 1,\ldots,n-1 \rbrace, \text{ or }\\
        \rho=0, \text{ and } \alpha \in \lbrace 3,\ldots,n-1 \rbrace, \text{ or }\\
        \rho\in \lbrace 1,\ldots \lfloor \frac{n-1}{4}\rfloor \rbrace, \text{ and } \alpha \in \lbrace 4\rho+1,\ldots,n-1 \rbrace .\\
    \end{cases}
\end{equation}
The associated lattice $\Lambda$ is then the saturation in $H^*_{alg}(S,\ZZ)$ of the lattice generated by $v$ and $a$, where
\[
a=\pm\left( \frac{X\pm \alpha}{2(n-1)},-Y\cH,\frac{X\mp \alpha}{2}  \right) \text{ for } X\equiv \mp \alpha\pmod{2(n-1)}.
\]

In our case, recall from Section \ref{sec: walls and chambers} that the value $t=4n-3$ is $n$-irregular with middle wall generated by $\cH_n-2\delta$, so by \cite[Lemma 3.6]{beri} it is associated to $(\rho,\alpha)=(-1,1)$ and $(X,Y)=(t,1)$. We conclude that the lattice $\Lambda$ associated to the flop is as follows.

\begin{definition}   
We let $\Lambda$ be the rank two, primitive sublattice 
$$\Lambda =\langle v,a\rangle \subset H^*_{alg}(S,\ZZ)$$
generated by $v=(1,0,-(n-1))$ and $a=-(2,-\cH,2n-1)$, which is nothing else than the Mukai vector of $\cU_2[1]$. 
The  associated Gram matrix is
$$ M=\begin{bmatrix}
2n-2 & 1 \\
1 & -2 
\end{bmatrix}.$$
We also let $w=v-a=(3,-\cH,n)$, so that  $w^2=2n-6$ and $(a,w)=3$.
\end{definition}

Instead of the generic stability condition $\sigma$, we also consider $\sigma_0$ a stability condition on $S^{[n]}$ which is generic 
on a wall corresponding to $\Lambda$ in the space of stability conditions. The condition that 
$Z_{\alpha,\beta}(v)$ and $Z_{\alpha,\beta}(a)$ be collinear yields that
\[
\alpha^2+\beta^2+\beta+\frac{n-1}{t}=0.
\]
Hence potential walls corresponding to $\Lambda$ correspond, in the half plane $(\beta,\alpha)\in \RR\times \RR_{> 0}$, to points in the semicircle of center $(-\frac{1}{2},0)$ and radius $\frac{1}{2\sqrt{t}}$. Objects whose Mukai vector lie in $\Lambda$ have the same phase with respect to $\sigma_0$ \cite[Definition 5.2]{bayermacri}.

\medskip
Similarly to the $n=3$ case, we can always choose $\sigma_0$ such that $a$ is effective. A priori this is asking less than \cite[Lemma 28]{invol1}, where we prove that  $\sigma_0$ can be chosen such that \textit{any} spherical class in $\Lambda$ is effective; but as it will turn out in the next subsection, effectivity is only important for $a$.

For simplicity, we suppose that $\sigma$ and $\sigma_0$ have the same associated heart $\Coh^\beta$.
Since the set of walls in $\rm{Stab}^\dagger(S)$ is locally finite, we can and will always consider $\sigma$ generic also with respect to $w$. Actually, we will consider $\sigma$ generic with respect any $v^{(i)}=v-(i+1)a$, $i\ge -1$,  of degree at least $-2$.
In the next subsection we will consider the moduli spaces of objects whose Mukai vector is one of those, since all these spaces naturally appear in the geometry of the contraction $c$.

\subsection{Jordan-H\"older filtrations}
\label{sec: base}
If $v^{(i)}=v-(i+1)a$, we get
$$v^{(i)}=(2i+3,-(i+1)\cH, (2i+1)n-i),$$
$$(v^{(i)})^2+2=2n-2(i+1)(i+2).$$
Thus by \cite[Theorem 3.6 (a)]{bayermacri}, the moduli space $\Sigma^{(i)}:=M_\sigma(v^{(i)})$ of $\sigma$-stable objects with Mukai vector $v^{(i)}$ is non-empty if and only if  $i\in \lbrace -1,\ldots, r\rbrace$, where  $r=max_{i\geq 0} \lbrace n-(i+1)(i+2)\geq 0 \rbrace$.

Each $v^{(i)}$ induces a decomposition in wall and chambers of $\Stab^\dagger(S)$.
Modifying the stability condition along a path from $\sigma$ to $\sigma_0$ that does not cross any wall, expect (possibly) on $\sigma_0$, produces a birational morphism $c^{(i)} :\Sigma^{(i)}\lra \Sigma^{(i)}_0$ with $\Sigma^{(i)}_0\textcolor{blue}:=M_{\sigma_0}(v^{(i)})$ normal \cite[Theorem 1.4]{bayermacri1}.  This morphism $c^{(i)}$ is defined by a suitably high multiple of a generator of $\theta((v^{(i)})^\perp\cap \Lambda)^\perp$ in $\NS(\Sigma^{(i)})$, where $\theta$ is the Mukai morphism \eqref{mukai morphism}, see \cite[Theorem 1.4]{bayermacri1} and \cite[proof of Theorem 12.1]{bayermacri}. 

We say that two objects parametrized by $\Sigma^{(i)}$ are \textit{S-equivalent for $\sigma_0$} if they have the same Jordan-H\"older factors with respect to $\sigma_0$ or equivalently, by \cite[Theorem 1.1]{bayermacri1}, if they have the same image in $\Sigma^{(i)}_0$. 
For $i=-1$ we recover $c:S^{[n]}\lra N$ and $v^{(0)}=w$. 
Studying all the $\Sigma^{(i)}$ together will provide a better understanding of the contraction $c$.

\begin{definition} The varieties $\Sigma^{(0)}$ and $\Sigma^{(0)}_0$ will also be denoted by $\Sigma$ and $\Sigma_0$ respectively, and the morphism $c^{(0)}$ by $\alpha:\Sigma\lra \Sigma_0$. The base of the contraction $c$  will be denoted  $B=c(I)$.
\end{definition}

We will need some technical lemmas. Recall from  \cite[Definition 5.4]{bayermacri} that a class $w\in \Lambda$ is positive (with respect to $v^{(i)}$) if $w^2\ge 0$ and $(w,v^{(i)})>0$.

\begin{lemma}\label{no-positive}
    $v^{(i)}$ is not the sum of two positive classes in $\Lambda$.
\end{lemma}

\proof 
Write $w_1=xv+ya$ and $w_2=(1-x)v-(i+1+y)a$. Imposing $w^2_1\geq 0$ and $w^2_2\geq 0$ gives
\[
4(n-1)x^2+4xy-4y^2=\left( (\sqrt{t}+1)x-2y \right)\left( (\sqrt{t}-1)x+2y \right)\geq 0,
\]
\[
\left( (\sqrt{t}+1)(1-x)+2(y+i+1) \right)\left( (\sqrt{t}-1)(1-x)-2(y+i+1) \right)\geq 0.
\]
This defines two lines $d_1,d_2$ passing through $(0,0)$ and two lines $\ell_1,\ell_2$ passing through $(1,-i-1)$, with $d_1$ parallel to $\ell_1$ and of positive slope, $d_2$ parallel to $\ell_2$ and of negative slope. These four lines are the sides of a parallelogram $\cP$ with vertices $(0,0)$ and $(1,-i-1)$; because of the condition $n\ge (i+1)(i+2)$, the latter is the easternmost 
point of $\cP$.  The conditions $w^2_1, w^2_2\geq 0$ are verified on the (closed)
parallelogram $\cP$ and two unbounded regions on its east and its west, which are cones with 
vertices $(0,0) $ and $(1,-i-1)$ respectively. 

Now we ask further that $(v^{(i)},w_1)=(2n-i-3)x+(2i+3)y>0$ and  $(v^{(i)},w_2)>(2n-i-3)(1-x)-(2i+3)(y+i+1)>0$. The first condition is verified above a line through $(0,0)$ with bigger slope than $d_2$, again because of the condition $n\ge (i+1)(i+2)$. The second condition  is verified below a parallel line through $(1,-i-1)$. As a result, the two unbounded regions are ruled out and we only remain with the parallelogram $\cP$. Since its westernmost point is $(0,0)$ and its easternmost point is $(1,-i-1)$, these are clearly its only two integral points.   But they correspond to  $w_1$ or $w_2$ being $0$, which is not allowed. \qed 

\begin{lemma}\label{no-spherical}
    Suppose  $s\in\Lambda$ is spherical and  such that  $0<(s,v^{(i)})\leq \frac{(v^{(i)})^2}{2}$. 
    Then either $s=a$ or $s=v^{(i+1)}$. This last case happens if and only if
    $n=m(m+1)$ for some integer $m>0$ and $i=m-2$.
\end{lemma}

\proof 
Suppose $s=xv+ya$; the sphericity condition means that 
\begin{equation}\label{hyperbola}
(n-1)x^2+xy-y^2=-1.    
\end{equation} 
The conditions that $0<(s,v^{(i)})\leq \frac{(v^{(i)})^2}{2}$ means that 
$$0< (2n-i-3)x+(2i+3)y\le n-1-(i+1)(i+2).$$
Letting $z=x-2y$, we rewrite these conditions as $z^2=(4n-3)x^2+4$ and 
\begin{equation}\label{eqeq}
0< 2(4n-3)x-2(2i+3)z\le (4n-3)-(2i+3)^2.
\end{equation}
If $x=0$, we get $y=1$, $z=-2$, so $s=a$. Suppose instead that $x\ge 1$; then the second inequality implies $z>0$. With $t=4n-3$ and $p=2i+3$, we get 
$$p< \frac{tx}{z} \quad \mathrm{and} \quad (p-z)^2\le \Delta^2:=z^2-(2x-1)t= t(x-1)^2+4.$$
Note that $z>\sqrt{t}x,$ hence $\frac{tx}{z}<\sqrt{t}<z$, so these two inequalities are always compatible (forgetting the condition that $p$ is integer) and reduce to 
$$z-\Delta\le p <\frac{tx}{z}.$$
Suppose that $x=1$, which implies that $4n+1$ is a square, hence $n=m(m+1)$ for some $m>0$; then $z=2m+1$, $\Delta=2$ and we are reduced to $2m-1\le p<2m+1-\frac{4}{2m+1}$; but then $p=2m-1$, 
$i=m-2$, and $(x,y)=(1,-m)$ is a solution, giving $s=v^{(i+1)}$. 

Now suppose that $x\ge 2$. In fact $x=2$ would mean $z^2=16n-8$ which is impossible, so we can suppose that $x\ge 3$. We claim that  
\begin{equation}\label{claim}
\delta:=\frac{tx}{z}-(z-\Delta)<\frac{1}{z}.
\end{equation}
But then,
$\frac{tx}{z}$ being at distance smaller than $\frac{1}{z}$ of the integer $p$, must be equal to $p$. In particular $z$ must divide $tx$, but $tx=kz$ yields $z(z-kx)=4$, which is impossible. 

There remains to prove (\ref{claim}). First observe that $z^2-\Delta^2 = (2x-1)t$, hence 
$$\delta = \frac{tx}{z}- \frac{t(2x-1)}{z+\Delta}=\frac{t(x\Delta-(x-1)z)}{z(z+\Delta)}.$$
In order to bound this expression, we use the identities
$$z=\sqrt{t}x+\frac{4}{z+\sqrt{t}x}, \qquad \Delta=\sqrt{t}(x-1)+\frac{4}{\Delta+\sqrt{t}(x-1)},$$
from which we deduce that 
$$x\Delta-(x-1)z = \frac{4x}{\Delta+\sqrt{t}(x-1)}-\frac{4(x-1)}{z+\sqrt{t}x}$$
Clearly $z$ is close to $\sqrt{t}x$; more precisely 
$$\frac{1}{2\sqrt{t}x}-\frac{1}{z+\sqrt{t}x}=\frac{z-\sqrt{t}x}{2\sqrt{t}x(z+\sqrt{t}x)}
=\frac{2}{\sqrt{t}x(z+\sqrt{t}x)^2}<\frac{1}{2t\sqrt{t}x^3}.$$
This yields the estimate
$$x\Delta-(x-1)z <\frac{2(2x-1)}{\sqrt{t}x(x-1)}+\frac{2(x-1)}{t\sqrt{t}x^3}.$$
Finally, we get the claimed upper bound 
$$z\delta<\frac{t(x\Delta-(x-1)z)}{(2x-1)\sqrt{t}}<\frac{2}{x(x-1)}+\frac{2(x-1)}{t(2x-1)x^3}
\le \frac{1}{3}+\frac{4}{135t}<1.$$
The case where $x<0$ is completely similar:
$z$ must also be negative, so we let $X=-x$ and $Z=-z$. The conditions become 
$$p>\frac{tX}{Z} \quad \mathrm{and} \quad (p+Z)^2\le \Delta^2:=Z^2+(2X+1)t= t(X+1)^2+4.$$
Then the very same discussion as before goes through, with all $-1$'s replaced by $+1$'s, and $X=1$ is no exception. \qed

\medskip 
The two previous lemmas allow  to apply \cite{bayermacri}, and we deduce the following geometric statement.

 \begin{lemma}\label{base irred}
    The base $B$ of the contraction $c$ is irreducible. 
    
    More generally, for any $i<r$, $c^{(i)}$ is a flopping contraction with an irreducible base  $B^{(i+1)}\subset \Sigma^{(i)}_0$. 
    For $i=r$, $c^{(i)}$ is an isomorphism.
\end{lemma}

\begin{proof}
Suppose first $(v^{(i)})^2\leq 0$; then $i=r$. If $(v^{(i)})^2=-2$, both $\Sigma^{(i)}$ and $\Sigma^{(i)}_0$ are points and there is nothing to prove. If $(v^{(i)})^2=0$, then $\Sigma^{(i)}$ is a K3 surface of Picard rank one by \cite[Theorem 3.6 (b)]{bayermacri}, in particular it does not contain any smooth rational curve. In this case, any birational transformation induced by wall-crossing is in fact an isomorphism, see \cite[Example 9.1]{bayermacri1}.

Now suppose that $(v^{(i)})^2>0$.
By definition $c^{(i)}$ is induced by a suitably high multiple of a class in $\NS(\Sigma^{(i)})$.
According to \cite[Theorem 5.7]{bayermacri}, it is a contraction if and only if there exists a special class in $\Lambda$, isotropic or spherical, subject to certain numerical conditions. There are several cases to consider. 

If a class ($s$ or $w$) as in \cite[Theorem 5.7 a)]{bayermacri} does exist, then together with $v^{(i)}$ it spans a rank two sublattice of $\Lambda$ with discriminant $d\in \lbrace 1,4,2(v^{(i)})^2\rbrace$. By \cite[XIV, (0.1)]{huybrechtsK3} this is not possible, since $\Lambda$ has discriminant $t$ which is always bigger than $d$. This excludes the possibility that $c^{(i)}$ is a divisorial contraction. 
So if $c^{(i)}$ is a contraction it must be a flopping one, subject to the conditions of \cite[Theorem 5.7 b)]{bayermacri}. There are two possibilities to consider, and they are addressed by Lemma \ref{no-positive} and Lemma \ref{no-spherical}.

\medskip
Every case in \cite[Theorem 5.7 b)]{bayermacri} yields a decomposition of $v^{(i)}$ as a sum of two effective classes in $\Lambda$ (a \textit{two-terms partition}), explicitly
described in the proof of \cite[Proposition 9.1]{bayermacri}. By Lemma \ref{no-positive}, the case of two positive classes never happens.
In the case involving a  spherical and effective class $s$, the decomposition is $v^{(i)}=s+(v^{(i)}-s)$, so Lemma \ref{no-spherical} shows that the only two-terms partition in our situation is $v^{(i)}=a+(v^{(i)}-a)=a+v^{(i+1)}$.
The case $s=v^{(i+1)}$ in Lemma \ref{no-spherical} corresponds to the situation in which $r=i+1$ and both terms of the decomposition are spherical.

As shown in \cite[Section 14]{bayermacri},
each two-terms partition of $v^{(i)}$ corresponds to an irreducible component of the base of the contraction $c^{(i)}$, so in our case the base of the contraction is irreducible.
\end{proof}

Conditions in \cite[Theorem 5.7]{bayermacri} describe all the possible Jordan-H\"older 
factors for objects corresponding to points in positive-dimensional fibers of a contraction. 
To lighten the exposition, from now on we only consider Jordan-H\"older filtrations giving rise to positive dimensional fibers, since those are the only ones that matter for our study of the contraction.

Let $A$ be a $\sigma$-stable object with Mukai vector $a$; since $a^2=-2$, $A$ is rigid.
Following the proof of \cite[Proposition 9.1]{bayermacri}, any extension of $A$ by an element $E^{(0)}\in \Sigma^{(0)}$ lies in a positive dimensional fiber of $c$.
That $E^{(0)}$ is the first element of a Jordan-H\"older filtration for $\cI_Z$ in $\Coh^\beta$ and, by the proof of Lemma \ref{base irred}, any such filtration for an element of $S^{[n]}$ starts in this way. Furthermore, again by Lemma \ref{base irred}, any element in a Jordan-H\"older filtration of $E^{(0)}$ lies in $\Sigma^{(1)}$ and so on.
In short, for any $Z\in S^{[n]}$,
any Jordan-H\"older filtration in $\Coh^\beta$ is of the form 
\begin{equation}\label{JOfiltr}
    0\subset E_Z^{(k)} \subset E_Z^{(k-1)} \subset \cdots \subset E^{(-1)}_Z=\cI_Z  
\end{equation}
for some $k\leq r$ and some objects $E^{(i)}_Z$ corresponding  to points of $\Sigma^{(i)}$; moreover,  $Z$ belongs to  $I$ if and only if this filtration is non trivial, that is,  $k\geq 0$. Here the inclusion sign means being a subobject in the category $\Coh^\beta$.

\medskip

Let us take a closer look to the Jordan-H\"older filtrations and factors, the latter being the cokernels of 
the inclusions in \eqref{JOfiltr}. We want to understand them in terms of coherent sheaves on $S$.
We will denote by $M_\cH(v^{(i)})$ the moduli space of $\cH$-semistable coherent sheaves with Mukai vector $v^{(i)}$.
Recall that $v^{(-1)}=v$, hence $M_\cH(v^{(-1)})=S^{[n]}$.

\begin{lemma}\label{de JOfiltr a JOmaps}
For $Z\in I$, set $\cE^{(-1)}=\cI_Z$. A filtration \eqref{JOfiltr} is equivalent to a 
collection of inclusions of sheaves
\begin{equation}\label{JOmaps}
\lbrace f^{(i)}:\cU_2\hookrightarrow \cE^{(i)} \rbrace_{i=0,\ldots,k}
\end{equation}
such that $\cE^{(i-1)}\in M_\cH(v^{(i-1)})$ is the cokernel of $f^{(i)}$.
Vice-versa, any such collection is the Jordan-H\"older filtration of some $Z\in I$. 
\end{lemma}

Beware that in \cite{invol1} we rather used the dual sheaves, which for $n=3$ were vector bundles.

\begin{proof}
Recall that $\Coh^\beta$ is obtained by extension-closure of two subcategories of the category of coherent sheaves, called $T^\beta$ and $F^\beta$.
By \cite[Proposition 2.4]{bayer}, a subobject of $\cE^{(i-1)}$ in $\Coh^\beta$ is equivalent to an exact sequence of coherent sheaves
    \[
0\lra \cK \lra \cE^{(i)} \stackrel{\nu}{\lra} \cE^{(i-1)} \lra \cQ\lra 0
\]
with $\cK\in F^\beta$, $\cQ\in T^\beta$ and $v(\cK)-v(\cQ)=-a$: we impose $v(\cQ)=(x,y\cH,z)$ and $v(\cK)=(2+x,(y-1)\cH,2n-1+z)$.
We proceed to prove that $\nu$ is necessarily surjective, in particular $\cQ=0$ and $v(\cK)=-a$, yielding $\cK\simeq \cU_2$ by semistability of $\cK$. This will prove the first part, since it will identify a subobject in $\Coh^\beta$ with an injective map $\cU_2\lra \cE^{(i)}$ with cokernel $\nu$.

The proof of \cite[Lemma 30]{invol1} is easily adapted as long as $x=0$, so that the rank of $\cK$ is two. This is the case for example when $i=0$. Suppose instead $i>0$ and $x>0$.
Since both $\cE^{(i)}$ and $\cE^{(i-1)}$ are semistable, the two inequalities 
\[
\left\{ \begin{aligned}
    \mu_\cH(\cK)\leq \mu_\cH(\cE^{(i)}) & \Longleftrightarrow \frac{y-1}{x+2}\leq -\frac{i+1}{2i+3} \\
    \mu_\cH(\cQ)\geq \mu_\cH(\cE^{(i)})& \Longleftrightarrow \frac{y}{x}\geq -\frac{i}{2i+1}
\end{aligned}
\right.
\]
must hold. 
The pair $(x,y)\in \ZZ^2$ should then lie in the triangle cut out in $\RR^2$ by $x=0$ and the two lines $g_1(x)=-\frac{i}{2i+1}x$ and $g_2(x)=-\frac{i+1}{2i+3}x+\frac{1}{2i+3}$, intersecting at $(2i+1,-i)$. Note that this triangle is contained in the slightly bigger triangle with integral vertices $(2i+1,-i), (0,0)$ and $(-2,1)$. This triangle has area $\frac{1}{2}$ so by Pick's formula the only integral points it contains are its three vertices.
The only possibility  with $x>0$ is therefore $(x,y)=(2i+1,-i)$. But in this case $\cK$ and $\cE^{(i)}$ have the same rank, so the image via $\nu$ in $\cE^{(i-1)}$ is a torsion sheaf, and has to be zero since $\cE^{(i-1)}$ is torsion free. 
This is impossible, since $\cK\in F^\beta$ and $\cE^{(i)}\in T^\beta$ cannot be isomorphic. This concludes the proof. \end{proof}

\begin{remark}\label{A}
As a consequence of the first part of the proof, the unique $\sigma$-stable object with Mukai vector $a$, that we denoted $A$ in the proof of Lemma \ref{base irred}, is in fact $\cU_2[1]$.
\end{remark}

\begin{lemma}\label{de JOfiltr a JOmaps second part}
With the same notations as before:
\begin{enumerate}
    \item any nonzero map $\cU_2\lra \cE^{(i)}$, with $\cE^{(i)}\in M_\cH(v^{(i)})$, is injective with cokernel in $M_\cH(v^{(i-1)})$;
    \item vice-versa,
    any nonzero morphism $\cE^{(i)}\lra \cE^{(i-1)}$, for $\cE^{(i)}\in M_\cH(v^{(i)})$ and $\cE^{(i-1)}\in M_\cH(v^{(i-1)})$, is surjective with kernel $\cU_2$; 
    \item  
    any non-trivial extension  of $\cE^{(i-1)}\in M_\cH(v^{(i-1)})$ by $\cU_2$ defines a point $\cE^{(i)}$ of $M_\cH(v^{(i)})$ which is $\cH$-stable.
\end{enumerate}
\end{lemma}

\begin{proof}
Let us prove (1). Let $\cK$ be the kernel of a non-zero map $\cU_2\lra \cE^{(i)}$. Suppose $\cK\neq 0$: since both $\cU_2$ and $\cE^{(i)}$ are torsion free, $\cK$ has rank one. Since $\cU_2$ is stable with slope $\mu_\cH(\cU_2)=-\frac{1}{2}$, $c_1(\cK)=-x\cH$ with $x>0$. The image of $\cU_2$ in $\cE^{(i)}$ has then rank one and first Chern class $(x-1)\cH$. This is absurd since the slope of $\cE^{(i)}$ is negative, so $\cK=0$. 

Then we check that the cokernel $ \cE^{(i-1)}$  must be torsion free. If not, let $\cT$ denote its torsion subsheaf, and $\tilde{\cT}$ the pullback of $\cT$ in 
$\cE^{(i)}$: as a subsheaf of the latter, this is a torsion-free sheaf, hence locally free in codimension one. Moreover, since $\cE^{(i)}$ is stable with negative first Chern class, $c_1(\tilde{\cT})<0$. But since $\cT$ is torsion, $c_1(\cT)\ge 0$, while the relation
$c_1(\cT)=c_1(\tilde{\cT})+\cH\le 0$ yields $c_1(\cT)=0$. This means that $\cT$ has finite support. Then $\cU_2\hookrightarrow \tilde{\cT}\hookrightarrow \tilde{\cT}^{\vee\vee}$
and this double dual is a vector bundle with the same first Chern class as $\cU_2$; so the inclusion of $\cU_2$ must be an isomorphism, hence also $\cU_2\simeq\tilde{\cT}$, hence $\cT=0$. 

Finally, suppose that $\cF$ is a rank $r$ destabilizing subsheaf of $ \cE^{(i-1)}$, with first Chern class $x\cH$, so that $x/r>-i/(2i+3)$. Its pullback $\tilde{\cF}$ in 
$ \cE^{(i)}$ has rank $r+2$ and first Chern class $(x-1)\cH$; since $ \cE^{(i)}$ is $\cH$-semistable, we deduce that $(x-1)/(r+2)\le -(i+1)/(2i+3)$. Hence 
$$\frac{1-ir}{2i+1}\le x\le \frac{1-(i+1)r}{2i+3},$$
which is impossible since  $\frac{1-(i+1)r}{2i+3}-\frac{1-ir}{2i+1}=-\frac{r+2}{(2i+1)(2i+3)}<0$.

\smallskip
We turn to (2). In light of the proof of the previous Lemma, we only need to show that the kernel $\cK$ of any non-zero $\cE^{(i)}\lra \cE^{(i-1)}$ is $\cH$--semistable. By the proof above, we know that $\cK$ has rank two and $c_1(\cK)=-\cH$. Consider any subsheaf $\cV$ of $\cK$ of rank one; since $\cV$ is also a subsheaf of $\cE^{(i)}$ which is $\cH$--semistable with negative slope, $c_1(\cV)=-x\cH$ for some integer $x>0$. But then $\mu_\cH(\cV)=-x<\mu_\cH(\cK)=-\frac{1}{2}$ and we are done.

\smallskip
To prove (3), first observe that an extension $\cF$ of $\cE^{(i-1)}$ by $\cU_2$ is necessarily torsion free with Mukai vector $v^{(i)}$, so we are left to prove $\cH$--semistability.
For $\cV$ a subsheaf of $\cF$, let $0\to \cK \to \cV\to \cQ\to 0$ be the induced extension, with $\cQ$ a subsheaf of $\cE^{(i-1)}$ of rank $r-k$, where $k\in \lbrace 0,1,2 \rbrace$  is the rank of $\cK$ and $r<2i+3$ the rank of $\cV$; we can suppose $r>k$. 

If $\cK$ has rank zero, being a subsheaf of $\cU_2$ it must be zero, hence
$\cV\simeq\cQ$. We must show this sheaf cannot destabilize $\cE^{(i)}$, knowing that it does not destabilize $\cE^{(i-1)}$. Suppose the contrary; if $c_1(\cV)=-x\cH$, this means that 
$$\frac{i}{2i+1} \le \frac{x}{r}\le\frac{i+1}{2i+3}.$$
The difference between the last two rational numbers can be written as $\frac{m}{r(2i+3)}$ for some positive integer $m$. On the other hand, the difference between the extreme two is $\frac{1}{(2i+1)(2i+3)}$. So the only possibility is that $m=1$ and $r=2i+1$, in which case we can suppose that $\cV=\cE^{(i-1)}$ and conclude that the extension splits, a contradiction. 

So suppose $k>0$, and let $c_1(\cV)=-x\cH$ and $c_1(\cK)=-y\cH$, with $x\ge 0$ and $y\ge 1$. Since $\cE^{(i-1)}$ is $\cH$-semistable, we have 
$\frac{x-y}{r-k}\geq \frac{i}{2i+1}$, in particular $x\geq y>0$. 
If $\cV$ has bigger or equal slope than $\cF$, then $\frac{x}{r}\leq \frac{i+1}{2i+3}$. In particular $x\le i$ since $r<2i+3$. Then the inequality 
$r-k\le 2(x-y)+\frac{x-y}{i}$ implies that in fact $r-k\le 2(x-y)$. Similarly, the inequality 
$r\ge 2x+\frac{x}{i+1}$ implies that in fact $r\ge 2x+1$. The two inequalities together yield $2y\le k-1$, in contradiction with the semistability of $\cU_2$.
\end{proof}

From now on, we refer also to a collection of maps as in \eqref{JOmaps} as a Jordan-H\"older filtration. By \cite[Definition and Corollary 6.1.5]{huy-lehn}, 
$$\chi(\cU_2,\cE^{(i)})=-(v(\cU_2),v(\cE^{(i)}))=(a,v^{(i)})=2i+3.$$ 
Since both sheaves are semistable and $\cE^{(i)}$ has bigger $\cH$-slope than $\cU_2$, we have 
$hom(\cE^{(i)},\cU_2)=0$. Using Serre duality we deduce that
\begin{equation}\label{hom et ext}
hom(\cU_2,\cE^{(i)})=(2i+3)+ext^1(\cE^{(i)},\cU_2).
\end{equation}
In particular $hom(\cU_2,\cE^{(i)})\geq 2i+3$.

As a nice consequence, we get a description of objects in $\Sigma^{(i)}$.

\begin{coro}\label{mod space}
    For $i\geq 0$, objects in $\Sigma^{(i)}$ are $\cH$-stable coherent sheaves with Mukai vector $v^{(i)}$. In particular, $\Sigma^{(i)}\simeq M_\cH(v^{(i)})$, and this moduli space parametrizes $\cH$-stable sheaves. 
\end{coro}

\begin{proof}
We prove the statement for $i=0$. The case where $i>0$ follows inductively with the same argument.
A priori, objects in $\Sigma^{(0)}=M_\sigma(v^{(0)})$ are two-term complexes; whenever $E\in \Sigma^{(0)}$ is a subobject of some $\cI_Z\in S^{[n]}$, seen in $\Coh^\beta$ as a complex concentrated in degree 0, then also $E$ is concentrated in degree zero, as in the proof of \cite[Proposition 2.4]{bayer}.
Now, \textit{any} $E\in \Sigma^{(0)}$ can be seen as a subobject of some $\cI_Z$, as it is clear from the construction of the indeterminacy locus $I$ in the proof of Lemma \ref{base irred}. So $E$ is always concentrated in degree zero and can be identified with an $\cH$-semistable sheaf $\cE$ of Mukai vector $v^{(0)}$.
As a consequence, $\Sigma^{(0)}$ embeds in $M_\cH(v^{(0)})$ by sending a complex 
to its only non-zero term.
The proof of Lemma \ref{de JOfiltr a JOmaps} translates the fact of being a subobject of some $\cI_Z$ for $E$ to sitting in an exact sequence between $\cU_2$ and the same $\cI_Z$ for the associated $\cE$.
Hence any such $\cE$ is actually $\cH$-stable by Lemma \ref{de JOfiltr a JOmaps second part} item (3).
By \cite[X.3. Theorem 3.10 and X.1. Corollary 1.5]{huybrechtsK3}, $M_\cH(v^{(0)})$ is irreducible of dimension $(v^{(0)})^2+2$ and $M_\sigma(v^{(0)})$ is projective of the same dimension \cite[Theorem 2.15 and Theorem 3.6]{bayermacri}, so the embedding is also surjective. 
\end{proof}

The Jordan-H\"older filtration is not unique, but unicity holds for Jordan-H\"older factors: the Jordan-H\"older factors of \eqref{JOfiltr}, except for the $r$-th one, have Mukai vector $a$ hence, by Remark \ref{A}, they are copies of  $\cU_2[1]$.
So the Jordan-H\"older factors are
\begin{equation}\label{JOfact}
[\cE^{(k)}_Z,\overbrace{\cU_2[1],\ldots,\cU_2[1]}^{k+1\,  times}].
\end{equation}

\medskip

We are ready to provide a description of the points of $I$.
Recall that, in $\Coh^\beta$, the ideal sheaves of points in the indeterminacy locus $I$ are extensions of $\cU_2[1]$ by an element of $\Sigma^{(0)}$ and two extensions are sent to the same point if and only if the associated $\cE^{(0)}$'s are S-equivalent for $\sigma_0$.

\begin{prop}\label{prop sequence for I}
A length $n$ subscheme $Z$ of $S$ defines a point of $I$ if and only if the sheaf $\cI_Z$
lies in a a short exact sequence
\begin{equation}\label{extension} 
0\lra \cU_2\stackrel{f}{\lra} \mathcal{E}\lra \cI_Z\lra 0
\end{equation}
for some $\cE$ corresponding to a point in $\Sigma\cong M_\cH(w)$.
Two subschemes $Z,Z'\in I$, with associated Jordan-H\"older filtrations as in \eqref{JOfiltr}, of respective lengths $k$ and $k'$, lie in the same fiber of $c$ if and only if $k=k'$ and $\cE^{(k)}_Z=\cE^{(k')}_{Z'}$. 
\end{prop}

\begin{proof}
Via Lemma \ref{de JOfiltr a JOmaps}, condition \eqref{extension} is equivalent to ask that any Jordan-H\"older filtration \eqref{JOfiltr} for $Z$ has length $k\geq 0$. The second part is clear by \eqref{JOfact}.

\end{proof}

\begin{remark}\label{rmq EN}
The exact sequence (\ref{extension}) can be seen as the Eagon-Northcott complex associated to the transpose of $f$, since $\cU_2$ and $\cE$ have the same determinant. Seen as a way to construct a sheaf from a finite scheme and an extension class, it is reminiscent of the Serre construction, which rather involves a Koszul complex. It could be interesting to explore this variant further. In particular, one could ask under which conditions the extension $\cE$ is a
vector bundle. This question has a very nice answer for the Serre construction on a K3 surface, see e.g. \cite[Theorem 3.13]{laz-linearseries}. In our specific situation, we expect the following statement: for $Z$ reduced in $I$, representing a collection $P_1,\ldots , P_n$ of planes in $V_{2n+1}$, whose span has dimension $2n-1$, then any subcollection of order $n-1$ should span a subspace of  $V_{2n+1}$ of the expected dimension $2n-2$. This is in agreement with the fact that for $n=3$, we only have vector bundles: if two points in $S$ represent planes that are not in general position, then the line they span is contained in $S$, which is not possible. 
\end{remark}

\begin{remark}\label{rmq recursive}
    The picture is recursive in nature and most of the statements have an analogue for any $i\geq 0$. One of them in particular will play an important r\^ole in the next section:
    for any $\cE^{(i)}\in \Sigma^{(i)}$, a Jordan-H\"older filtration is a truncation of \eqref{JOfiltr}, from $i$ to some $k\leq r$.
    From that, and using Lemma \ref{de JOfiltr a JOmaps} and Lemma \ref{de JOfiltr a JOmaps second part} one shows easily that (as it is for $i=-1$ in Proposition \ref{prop sequence for I}), for any $0\leq i<r$,  $\cE^{(i)}$ lies in the exceptional locus of $c^{(i)}$ if and only if it lies in a short exact sequence
    \[
    0\lra \cU_2\lra \cE^{(i+1)}\lra \cE^{(i)}\lra 0
    \]
    for some $\cE^{(i+1)}\in \Sigma^{(i+1)}$. Moreover, two elements of $\Sigma^{(i)}$ lie in the same fiber of $c^{(i)}$ if and only if their last Jordan-H\"older factor is the same.
\end{remark}

Now that we have a description in terms of sheaves for objects in $\Sigma^{(i)}$ and for Jordan-H\"older filtrations and factors, we will  use it to describe natural stratifications on $N$ and $S^{[n]}$.

\subsection{Stratification of the base of the contraction and the indeterminacy locus}
There is a natural stratification of $S^{[n]}$, introduced in 
\cite[Lemma 14.1]{bayermacri}.

\begin{definition}\label{def strates Sn}
We denote by $I_k\subset S^{[n]}$ the locus of schemes $Z$ whose Jordan-H\"older filtrations have length exactly $k+1$.
\end{definition}

Via $c$, the stratification on $S^{[n]}$ induces a stratification on $N$,
\begin{equation}\label{strates N}
\emptyset =\bar{B}^{(r+1)}\subset \bar{B}^{(r)}\subset \ldots \subset \bar{B}^{(0)}=B^{(0)}=B\subset B^{(-1)}=N    
\end{equation}
with $r+2$ strata, where each $\bar{B}^{(k)}$ is the closure of the image of $I_k$ via $c$ and 
$c(I_k)=\bar{B}^{(k)}\smallsetminus \bar{B}^{(k+1)}$.
This space parametrizes S-equivalence classes of subschemes $Z$ for which any Jordan-H\"older filtration \eqref{JOfiltr} has length $k+1$.
The stratifications depend in a fundamental way on \eqref{hom et ext}.

\begin{definition}
    We denote by $I^{(i)}\subset \Sigma^{(i)}$ the exceptional locus of $c^{(i)}$ and $U^{(i)}:=\Sigma^{(i)}\smallsetminus I^{(i)}$.
\end{definition}

The open subset $U^{(i)}$ is the space of strictly $\sigma_0$-stable elements of $\Sigma^{(i)}$. It is equal to the whole $\Sigma^{(i)}$ if and only if $i=r$.

\begin{lemma}\label{aumoins3 debut}
For any $i\geq -1$, $hom(\cU_2,\cE^{(i)})= 2i+3$ if and only if $\cE^{(i)}$ 
is the only element in its S-equivalence class for $\sigma_0$, or equivalently $\cE^{(i)}\in U^{(i)}$.
If $\cE^{(i)}$ lies in $U^{(i)}$, then $ext^1(\cE^{(i)},\cU_2)=0$.
\end{lemma}

\begin{proof}
We prove the statement for $i=-1$, the other cases are proved in the exact same way because of the observations in Remark \ref{rmq recursive}.
An element of $S^{[n]}$ lies in $I$ if and only if the associated ideal sheaf lies in a short exact sequence \eqref{extension}. In particular, this is equivalent to ask $ext^1(\cI_Z,\cU_2)>0$, or $hom(\cU_2,\cE^{(i)})> 2i+3$ by \eqref{hom et ext}.

Similarly, if $ext^1(\cI_Z,\cU_2)\neq 0$ by Lemma \ref{de JOfiltr a JOmaps second part} there is an inclusion $\cU_2\hookrightarrow \cE^{(0)}$ for some $\cE^{(0)}$, whose cokernel is $\cI_Z$, so $Z\in I$.
\end{proof}

By \eqref{extension} and Remark \ref{rmq recursive}, together with Lemma \ref{de JOfiltr a JOmaps second part}, for any $i\geq 0$ and $\cE^{(i)}$ there is a surjective morphism from $\PP(\mathrm{Hom} (\cU_2,\cE^{(i)}))$ to the space of elements of $\Sigma^{(i-1)}$ which admit a Jordan-H\"older filtration as in \eqref{JOmaps} whose first term $f^{(0)}$ lies in $\mathrm{Hom}(\cU_2,\cE^{(i)})$. Although that space is not necessarily a fiber of $c^{(i-1)}$, it is clearly contained in one such fiber, which we denote $F_{\cE^{(i)}}$.
This allows us to define a surjective continuous map
\begin{equation}
    \eta^{(i)}:\Sigma^{(i)}_0 \lra B^{(i)}
\end{equation}
sending $c^{(i)}(\cE^{(i)})$ to $c^{(i-1)}(F_{\cE^{(i)}})$.
By the observation on Jordan-H\"older filtrations made in Remark \ref{rmq recursive}, $\eta^{(i)}$ is bijective, but it is not clear a priori that it is a morphism.

In order to prove it, 
for each $i\ge 0$ we consider the universal sheaf $\cF^{(i)}$ on $S\times\Sigma^{(i)}$ (depending on $n$ and $i$, we may need to consider the \textit{twisted} universal sheaf instead, see \cite[X.2.2 (ii)]{huybrechtsK3}).
We denote by $p_1$ and $p_2$ the projections from $S\times \Sigma^{(i)}$ to the two factors and by $$\cG^{(i)}= p_{2*}Hom(p_1^*\cU_2,\cF^{(i)})$$ 
the (twisted) sheaf on $\Sigma^{(i)}$ obtained by push-forward. As a consequence of Lemma \ref{aumoins3 debut}, $\cG^{(i)}$ is locally free of rank $2i+3$ on $U^{(i)}$.
We consider the projectivization of the restriction $\PP_{U^{(i)}}(\cG^{(i)})$.
By the same reasoning as above, there is a surjective morphism 
\begin{equation}\label{eq desirée}
\PP_{U^{(i)}}(\cG^{(i)})\lra I^{(i-1)}_0,    
\end{equation}
where $I^{(i-1)}_0\subset I^{(i-1)}$ is the locus of objects whose Jordan-H\"older filtrations have minimal length.
This morphism descends to a morphism $\eta^{(i)}_0:c(U^{(i)})\lra c^{(i-1)}(I^{(i-1)}_0)$; by construction, it coincides with $\eta^{(i)}$ over the open $c(U^{(i)})$. We can now prove the following

\begin{lemma}
    $\eta^{(i)}$ is a bijective normalization morphism.
\end{lemma}

\begin{proof}
We already know that $\eta^{(i)}$ is bijective, we just need to prove that it is a morphism.
In $\Sigma^{(i)}_0\times B^{(i)}$, we consider the graphs $\Gamma$ and $\Gamma_0$ of $\eta^{(i)}$ and $\eta^{(i)}_0$, respectively; we denote by $p$ and $q$ the two projections, to $\Sigma^{(i)}_0$ and $B^{(i)}$, respectively.

The graph $\Gamma$ is the closure of $\Gamma_0$,
since $\eta^{(i)}$ restricts to $\eta^{(i)}_0$ on a dense open subset; in particular $\Gamma$ is a projective variety, since $\eta^{(i)}_0$ is a morphism.
As $\Gamma$ is a  graph, the restriction $p_{|\Gamma}:\Gamma\lra \Sigma^{(i)}_0$ is 
bijective. 
So $p_{|\Gamma}$ is a birational, proper and bijective map; by Zariski's Main Theorem, 
it is an isomorphism, $\Sigma^{(i)}_0$ being normal. But then we are done, since $\eta^{(i)}=q\circ (p_{|\Gamma})^{-1}$. 
\end{proof}

The normalization $\eta^{(0)}$ will also be denoted by $\eta$.

\smallskip
For any $i\geq 0$, the base of the $i$-th contraction $B^{(i)}$ admits an injective map to $B^{(i-1)}$, via the inclusion in $\Sigma^{(i-1)}_0$ and the bijective normalization $\eta^{(i-1)}$. Repeating from $r$ to $0$, we obtain injective maps $\tau^{(i)}:B^{(i)}\lra N$; for each $i$, the image is the $i$-th stratum
$\bar{B}^{(i)}$ defined in \eqref{strates N}.
Note that $\tau^{(0)}$ is just the identity $\bar{B}^{(0)}=B^{(0)}$. We let $\eta^{(-1)}=\tau^{(-1)}=id_N$.
Putting everything together, we get the following diagram.

\tiny
\begin{equation}\label{big beautiful diagram}
    \begin{tikzcd}
\Sigma^{(r)}\ar[d,"\eta^{(r)}"]= \Sigma^{(r)}_0 & \Sigma^{(r-1)}\ar[d,"c^{(r-1)}"]  & & & & \\
B^{(r)}\ar[r,hook]\ar[dddd,"\tau^{(r)}"]&  \Sigma^{(r-1)}_0\ar[d,"\eta^{(r-1)}"] & & & & \\
 & B^{(r-1)}\ar[r,hook]\ar[ddd,"\tau^{(r-1)}"] & \ldots & \Sigma^{(1)}\ar[d,"c^{(1)}"] & & \\
& &   \ldots \ar[r,hook]& \Sigma^{(1)}_0 \ar[d,"\eta^{(1)}"] & \Sigma^{(0)}=\Sigma\ar[d,"c^{(0)}=\alpha"] &  \\
& & &            B^{(1)}\ar[d,"\tau^{(1)}"] \ar[r,hook]    & \Sigma^{(0)}_0=\Sigma_0 \ar[d,"\eta^{(0)}=\eta"] & \Sigma^{(-1)}=S^{[n]}\ar[d,"c^{(-1)}=c"] \\
 \bar{B}^{(r)}\ar[r,hook] & \bar{B}^{(r-1)}\ar[r,hook] &\ldots \ar[r,hook]                  & \bar{B}^{(1)} \ar[r,hook] &   \bar{B}^{(0)}=B^{(0)}=B \ar[r,hook] & \Sigma^{(-1)}_0=N. \\
    \end{tikzcd}
\end{equation}
\normalsize

It is then clear that the trivial identity $\bar{B}^{(-1)}=N$ can be generalized to any stratum.

\begin{coro}\label{bijective}
    For any $i\geq 0$, $\bar{B}^{(i)}\subset N$ is isomorphic to $\Sigma^{(i)}_0$, up to a bijective normalization. 
\end{coro}

\begin{definition} 
For $\cE^{(i)}\in \Sigma^{(i)}$, we denote by $[\cE^{(i)}]$ its image in $\bar{B}^{(i)}\subset N$ via $\tau^{(i)}\circ \eta^{(i)}\circ c^{(i)}$ (see \eqref{big beautiful diagram}).
\end{definition}

For $\cE^{(i)}\in \Sigma^{(i)}$ and $\cE^{(i+1)}\in \Sigma^{(i+1)}$, it may happen that $[\cE^{(i)}]=[\cE^{(i+1)}]$ when $\cE^{(i)}\in B^{(i+1)}$. For example, by definition $c(Z)=[\cI_Z]$ for any $Z$, but this is the only possible class $\cE$ such that $c(Z)=[\cE]$ if and only if $Z\notin I$.

The next Proposition is a good summary of the relation between the contraction $c$ and the stratification of $N$, in light of \eqref{big beautiful diagram} and the Definition above.

\begin{prop}\label{image of c}
For $Z\in I$ with Jordan-H\"older factors \eqref{JOfact},
\[c(Z)=[\cE^{(k)}_Z]\in \bar{B}^{(k)}\smallsetminus \bar{B}^{(k+1)}.\]
More  generally, $c(Z)=[\cE^{(i)}]$ for $\cE^{(i)}\in \Sigma^{(i)}$ and $i\leq k$
if and only if there exists a Jordan-H\"older filtration \eqref{JOmaps} for $Z$ in which $\cE^{(i)}$ appears.
\end{prop}

\begin{proof}
This is a restatement of the proof of Lemma \ref{base irred} in the language of sheaves provided by the statements in Subsection \ref{sec: base}, using the commutativity of all the small diagrams in \eqref{big beautiful diagram}. 
\end{proof}

By Proposition \ref{image of c}, for any $Z$ for which $\cI_Z$ lies in a sequence \eqref{extension}, we have $c(Z)=[\cE]\in B$. The depth of $c(Z)$ in the stratification $\lbrace  \bar{B}^{(i)} \rbrace_i$, hence of $Z$ in $\lbrace I_k \rbrace_i$ depends on \eqref{hom et ext} by Lemma \ref{aumoins3 debut}.
This will be crucial in the next subsection. 

\begin{remark} 
Proving that $B^{(i)}$ is normal for any $i$ would simplify a lot the picture. In that case, we would have $\bar{B}^{(i)}=B^{(i)}=\Sigma^{(i)}_0$ and the diagram \eqref{big beautiful diagram} would reduce to
\begin{equation}
    \begin{tikzcd}
        &  \Sigma^{(r-1)} \ar[d,"c^{(r-1)}"] & \ldots &  \Sigma^{(0)}\ar[d,"c^{(0)}"]  & \Sigma^{(-1)}=S^{[n]}\ar[d,"c^{(-1)}=c"]\\
        \Sigma^{(r)}=\Sigma^{(r)}_0\ar[r,hook] &  \Sigma^{(r-1)}_0 \ar[r,hook] & \ldots \ar[r,hook]&  \Sigma^{(0)}_0 \ar[r,hook] & \Sigma^{(-1)}_0=N.
    \end{tikzcd}
\end{equation} 
This is the case for $n=2$ and $n=3$: conjecturally, for any $n$ and any $i\geq 0$,  $B^{(i)}$ should be smooth away from $B^{(i+1)}$.
\end{remark}

\begin{remark}
For $n=2,3$, studied in \cite{ogradydual} and \cite{invol1} respectively, the situation is much simpler. In these cases $r=0$ and $\Sigma^{(0)}_0=\Sigma^{(0)}=\Sigma$ is respectively a single point or a smooth K3 surface. The diagram \eqref{big beautiful diagram} becomes 
\[
\begin{tikzcd}
    & S^{[n]} \ar[d,"c"] \\
\Sigma\ar[r,hook] & N.
\end{tikzcd}
\]
Up to $n=5$ we also need to take into account the normalization $\eta:\Sigma\lra B$. 
Starting from $n=6$, the stratification of $N$ has at least three strata since $r\geq 1$, so $B$ is always singular since $\Sigma^{(0)}$ is.

For $n=6$ then $\Sigma^{(1)}=\bar{B}^{(1)}$ is a singleton, for $n=7$ the stratum $\bar{B}^{(1)}$ is a surface of which $\Sigma^{(1)}$, a K3 surface, is a bijective normalization, and so on. More generally, the deeper stratum is a point if and only if $n=(r+1)(r+2)$, a surface whenever $n=(r+1)(r+2)+1$, etc.
\end{remark}

\subsection{Structure of the indeterminacy locus}
\label{descr of I}

We are now ready to relate the locus $J\subset S^{[n]}$ of subschemes generating a non-maximal linear space in $V_{2n+1}$ with the indeterminacy locus $I$.

\begin{theorem}\label{thm I=J}
$I=J$.
\end{theorem}

\begin{proof}
By Lemma \ref{aumoins3 debut}, $Z\in I$
if and only if $hom(\cU_2,\cI_Z)>1$. By Lemma \ref{J-BrillNoether}, this is equivalent to asking that  $Z\in J$.
\end{proof}

We can be much more precise.

\begin{definition}
For $k\geq -1$, let $J_k$ be the locus of subschemes generating a linear space of codimension $k+2$ in $V_{2n+1}$.
\end{definition}

The $J_k$'s form a stratification of $S^{[n]}$ and of $I=\bigcup_{k\geq 0} J_k$.
We defined another stratification $\lbrace I_k\rbrace_{k=-1\ldots r}$, see Definition \ref{def strates Sn}. By Theorem \ref{thm I=J} we have $I_{-1}=J_{-1}$.  

\begin{prop}\label{Bk et Jk}
    Consider $Z\in I$ such that $c(Z)\in \bar{B}^{(k)}\smallsetminus \bar{B}^{(k+1)}$ for some $k\geq 0$. Then $Z$ generates in $V_{2n+1}$ a linear space of codimension $k+2$. Equivalently, for any $k\geq -1$ we have
    \[ I_k=J_k. \]
\end{prop}

\begin{proof}
Consider $Z\in I$.
We prove that $Z\in I_k$ if and only if $k$ is maximal such that the ideal sheaf of $Z$ fits into an exact sequence
\begin{equation}\label{seq U2 k+1}
    0 \lra \cU_2^{\oplus(k+1)} \stackrel{g}{\lra} \cE^{(k)}\lra\cI_Z\lra 0,
\end{equation}
with $\cE^{(k)}\in \Sigma^{(k)}$. Equivalently, we show that, if $\cI_Z$ fits into a sequence \eqref{seq U2 k+1}, then we can produce a collection of length $k+1$ of inclusions like in \eqref{JOmaps}, and vice-versa.

Take $Z$ whose ideal sheaf fits in a short exact sequence \eqref{seq U2 k+1}
for some $\cE^{(k)}\in \Sigma^{(k)}$. 
For $i=0,\ldots,k$, we denote by $j_i:\cU_2\lra \cU_2^{\oplus(k+1)}$ the $i$-th inclusion. 
We define $f^{(k)}=g\circ j_k:\cU_2\lra \cE^{(k)}$ and we denote by $q^{(k-1)}:\cE^{(k)}\lra \cE^{(k-1)}$ its cokernel, which by construction admits a surjective morphism $h^{(k-1)}:\cE^{(k-1)}\lra \cI_Z$. Then we define $$f^{(k-1)}=q^{(k-1)}\circ g\circ j_{k-1}:\cU_2\lra  \cE^{(k-1)}.$$
Proceeding in this way,  we construct inductively a collection of injective morphisms 
$f^{(i)}:\cU_2\lra \cE^{(i)}$ for $i=0,\ldots,k$, and of surjective morphisms $h^{(i)}:\cE^{(i)}\lra \cI_Z$ for $i=-1,\ldots, k$.
Observe that, by construction, the kernel of $h^{(i)}$ is a sum of copies of $\cU_2$, in particular it is torsion free, hence $\cE^{(i)}$ is also torsion free.
Then $h^{(0)}$ yields an isomorphism $\cE^{(-1)}\simeq \cI_Z$ and by Lemma \ref{de JOfiltr a JOmaps second part} we have $\cE^{(i)}\in \Sigma^{(i)}$ for any $i$. It is then clear that
the collection $\lbrace f^{(i)}:\cU_2\lra \cE^{(i)}  \rbrace_{i=0,\ldots,k}$ obtained in this way is is in the form \eqref{JOmaps} with $\cE^{(-1)}=\cI_Z$.

Conversely, 
consider $Z$ with a collection of maps as in \eqref{JOmaps}, of length $k+1$. We construct inductively a collection of injective morphisms $g_i:\cU_2^{\oplus (i+1)}\lra \cE^{(i)}$, starting with  $g_0=f^{(0)}$. Consider the short exact sequence 
\begin{equation}\label{hom(i) et Hom(i-1)}
0\to \cU_2\xrightarrow{f^{(i)}}\cE^{(i+1)}\xrightarrow{q^{(i)}}\cE^{(i)}\to 0.    
\end{equation}
Since $\cU_2^\vee$ is spherical, by twisting this sequence by $\cU_2^\vee$ and taking cohomology, we see that the morphism $Hom(\cU_2, \cE^{(i+1)})\ra Hom(\cU_2, \cE^{(i)})$ is surjective. Then it is easy to see that, $g_{i-1}$ being given,  there exists  a lifting $g_{i}$ fitting in a commutative diagram
$$\begin{CD}
  @.   @. 0   @. 0   @.\\
    @.   @. @VVV        @VVV \\
  0  @>>> \cU_2   @>>> \cU_2^{\oplus (i+1)}   @>>> \cU_2^{\oplus i}   @>>> 0 \\
    @.  @|  @VVg_iV        @VVg_{i-1}V  \\
  0  @>>> \cU_2   @>>> \cE^{(i)}  @>>> \cE^{(i-1)}   @>>> 0 \\
   @.   @. @VVV        @VVV   \\
   @. @. \cI_Z  @= \cI_Z   @. \\
   @.   @. @VVV        @VVV  \\
     @.  @. 0   @. 0   @.  
\end{CD}$$
Then $g=g_k$ has the required properties.
\medskip

Having obtained \eqref{seq U2 k+1} for $Z\in I_k$, we twist this sequence by $\cU_2^\vee$ and we take cohomology. Since $\cU_2^\vee$ is spherical we get a short exact sequence 
\[
0\lra \CC^{k+1} \lra Hom(\cU_2,\cE^{(k)})\lra H^0(\cI_Z\otimes \cU_2^\vee)\lra 0.
\]
Now, $[\cE^{(k)}]\in \bar{B}^{(k)}\smallsetminus \bar{B}^{(k+1)}$ since $Z\in I_k$. 
By Lemma \ref{aumoins3 debut} this implies $hom(\cU_2,\cE^{(k)})=2k+3$, so $h^0(\cI_Z\otimes \cU_2^\vee)=2k+3-(k+1)=k+2$.
\end{proof}

Since $c$ is semismall \cite[Lemma 2.11]{kaledin}, the dimension of the fiber of a general point of $\bar{B}^{(k)}$ is at most $(k+1)(k+2)$.

The first stratum $k=0$ is easy to study:
for $[\cE]\in B\smallsetminus \bar{B}^{(1)}$, hence $\cE\in U^{(0)}$, \eqref{eq desirée} restricts by Lemma \ref{aumoins3 debut} to a
surjective morphism
\begin{equation}\label{1ere strate}
\PP(Hom(\cU_2,\cE))\lra c^{-1}([\cE])     
\end{equation}
sending the class of $f^{(0)}$ to $coker(f^{(0)})=\cI_Z$ for some $Z$. 
But $f^{(0)}$ is unique up to scalar, since by hypothesis $\cE\in U^{(0)}$, or equivalently $\cE$ is unique in its S-equivalence class, so $ext^1(\cE,\cU_2)=1$.
Hence the fiber over a point of $B\smallsetminus\bar{B}^{(1)}$  is isomorphic to $\PP(Hom(\cU_2,\cE))\simeq\PP^2$.

\begin{coro}\label{structure I}
    $I$ is irreducible, and birational to a $\PP^2$-fibration over $\Sigma$.
\end{coro}

\begin{proof}
    Since $I_0$ is a $\PP^2$-fibration over an irreducible base, it is irreducible. Then $\overline{I_0}=\cup_{k=-1,\ldots r}I_r=I$ by \cite[Lemma 14.1]{bayermacri}.
\end{proof}

The description of $I_0$ is essentially \cite[Lemma 14.2]{bayermacri}.
For deeper strata the description of the fibers is more complicated, but we are able to prove that they always have the expected dimension.

\begin{lemma}\label{dim fibre}
For $k\geq 0$, the fiber of $c$ over a point of 
$\bar{B}^{(k)}\smallsetminus \bar{B}^{(k+1)}$ is unirational (hence irreducible) of dimension $(k+1)(k+2)$.
\end{lemma}

\proof 
Any Jordan-H\"older filtration \ref{JOmaps} comes with a sequence of sheaves $\underline{\cE}=(\cE^{(k)},\cE^{(k-1)},\ldots,\cE^{(0)})$. We 
fix $[\cE^{(k)}]\in \bar{B}^{(k)}\smallsetminus \bar{B}^{(k+1)}$ and consider the space 
\[
\cZ=\Big\{
\underline{\cE}=(\cE^{(k)},\cE^{(k-1)},\ldots,\cE^{(0)}) \,|\, [\cE^{(i)}]=[\cE^{(k)}] \text{ for all }i
\Big\},
\]
parametrizing the sheaves in Jordan-H\"older filtrations with the same Jordan-H\"older factors \eqref{JOfact}, see Proposition \ref{image of c}. 
By Corollary \ref{bijective}, $\cE^{(k)}$ is uniquely determined by $[\cE^{(k)}]$; then 
$\cE^{(k-1)}$ is the cokernel of some nonzero morphism $f_k: \cU_2\lra \cE^{(k)}$, and so on.
Thus the space of Jordan-H\"older filtrations  $\underline{\cE}$ that lie in $\cZ$ admits a surjective morphism to the fiber $c^{-1}([\cE^{(k)}])$, sending $(f^{(k)},\ldots,f^{(0)})$ to $Z$ such that $coker(f^{(0)})=\cI_Z$.
Note that arguing as in the proof of Proposition \ref{Bk et Jk}, from \eqref{hom(i) et Hom(i-1)} we deduce that 
$Hom(\cU_2,\cE^{(k-1)})\simeq Hom(\cU_2,\cE^{(k)})/\CC f_k$ has fixed dimension, and so on. We therefore dominate 
the fiber $c^{-1}([\cE^{(k)}])$ by a tower of projective fibrations, which  already ensures that it is unirational. 

In order to compute its dimension, we observe that 
a Jordan-H\"older filtration $(f^{(k)},\ldots,f^{(0)})$ is sent to a given finite scheme $Z\in I$ when $f^{(0)}$
defines an extension $\cE^{(0)}$ of $\cI_Z$ by $\cU_2$, $f^{(1)}$
defines an extension $\cE^{(1)}$ of  $\cE^{(0)}$ by $\cU_2$, and so on. Again, the corresponding extension spaces have fixed dimensions, so that the filtrations with given image $Z$ are also parametrized by a tower of projective fibrations. We can then compute the dimension of the fiber $c^{-1}([\cE^{(k)}])$ as the difference between the dimensions of these two towers, which yields 
\[
\sum_{i=0}^k (hom(\cU_2,\cE^{(i)})-1)
-\sum_{i=0}^k(ext^1(\cE^{(i-1)},\cU_2)-1)\stackrel{\text{by }\eqref{hom et ext}}{=}
\]
\[
=ext^1(\cE^{(k)},\cU_2)+\left(\sum_{i=0}^k 2i+3 \right)-ext^1(\cI_Z,\cU_2)\stackrel{\text{Lemma }\ref{aumoins3 debut}}{=} 
\]
\[
=2\left(\sum_{i=0}^k i\right) +3(k+1)-ext^1(\cI_Z,\cU_2)=k(k+1)+3(k+1)-ext^1(\cI_Z,\cU_2).
\]
By Serre duality, $ext^1(\cI_Z,\cU_2)=h^1(\cI_Z\otimes \cU_2^\vee)$ and the latter is $h^0(\cI_Z\otimes \cU_2^\vee)-1$, as we saw in Lemma \ref{J-BrillNoether}. Finally, we are able to conclude by Proposition \ref{Bk et Jk}, since
\[
k^2+4k+4-h^0(\cI_Z\otimes \cU_2^\vee)=k^2+4k+4-(k+2)=(k+1)(k+2). \qedhere
\]

\medskip
We close this section by collecting our results about the indeterminacy locus $I$ of $\varphi$. Recall that $r=max_{i\geq 0} \lbrace n-(i+1)(i+2)\geq 0 \rbrace$.

\begin{theorem}\label{descr I}
The indeterminacy locus $I\subset S^{[n]}$ of the birational involution $\varphi$ coincides with  the locus $J$ of subschemes generating a linear supspace in $V_{2n+1}$ of codimension at least two. Moreover, $I$ is irreducible and is contracted by $c$ to the base $B\subset N$, admitting a stratification
\[
\emptyset\subset \bar{B}^{(r)}\subset \cdots \subset\bar{B}^{(1)}\subset \bar{B}^{(0)}=B 
\]
where $\bar{B}^{(k)}$ has codimension $2(k+1)(k+2)$ in $N$. 

For each $k\geq 0$, the restriction of $c$ to $\bar{B}^{(k)}\smallsetminus \bar{B}^{(k+1)}$ is a fibration with unirational fibers of dimension $(k+1)(k+2)$. A scheme $Z\in I$ is sent by $c$  to $\bar{B}^{(k)}\smallsetminus \bar{B}^{(k+1)}$ if and only if it generates a subspace of $V_{2n+1}$ of codimension exactly $k+2$.
\end{theorem}

We can also translate our result in a Brill-Noether like statement. Recall that we denote by $v^{(k)}$ the Mukai vector $v-(k+1)a$, see Section \ref{sec: base}.

\begin{theorem}\label{descr I - BN}
Let $(S,\cH)$ be a very general K3 surface of degree $8n-6$, embedded in $G(2,V_{2n+1})$ via the Mukai bundle of Mukai vector $(2,\cH,2n-1)$. The space $J_k$ of degree $n$ finite subschemes of $S$ generating a linear subspace of $V_{2n+1}$ of codimension  $k+2$ is non-empty if and only if $(k+1)(k+2)\leq n$. In that case, it has dimension $2n-(k+1)(k+2)$ and is a fibration 
over an open subset of the moduli space $M_\cH(v^{(k)})$ of $\cH$-stable coherent sheaves on $S$ with Mukai vector $v^{(k)}$, with unirational fibers of dimension $(k+1)(k+2)$. The complement of this open subset has codimension $2(k+2)$ in $M_\cH(v^{(k)})$. 
\end{theorem}

We already knew that $I$ is a $\PP^2$-fibration for $n=2,3$. The same statement holds for $n=4,5$, but is no longer true for $n\geq 6$. For $n=6$, $I$ is a $\PP^2$-fibration
except for a single fiber of dimension $6$ corresponding to the subschemes generating a linear space of codimension $3$. In general, whenever $n=(r+1)(r+2)$, the last stratum $\bar{B}^{(r)}\cong M_\sigma(v^{(r)})$ is a single point. Hence the subschemes generating a linear space of codimension $r+2$ are parametrized by a unirational variety of dimension $(r+1)(r+2)$, 
which is contracted to a point by $c$.

\medskip
\begin{remark}
For $n=3$, we proved in  \cite{invol1} that $\varphi$ is a Mukai flop. In particular it can be resolved by blowing-up the indeterminacy locus $I$ and contracting it again in other directions. We expect something similar, but more complicated for bigger $n$: we would need to 
blow-up first $I_{r-1}$ along $I_r$, then the bigger strata recursively, in order to finally resolve the indeterminacies of $\varphi$, which should be a kind of stratified Mukai flop.

One more time, we are inspired by Beauville's involution. In that framework, Markman described a Brill-Noether stratification of $S^{[n]}$ and proceeded to resolve the birational map by blowing up the strata recursively, see \cite{markman-BN} and \cite[Section 4]{ogradyinvolutions}.
\end{remark}

\section{The variety $\Sigma$}\label{sec: Sigma}

Let us study more closely the hyperK\"ahler manifold $\Sigma$ 
that desingularizes the base $B$ of the contraction $c$, see
 \eqref{big beautiful diagram}; we assume again for simplicity that $C_n=1$. The morphism $\Sigma\lra B$  factors into
$\alpha:\Sigma\lra \Sigma_0$ and  a bijective normalization $\eta: \Sigma_0\lra B$.
We will be mostly interested in the first map. By Lemma \ref{base irred}, $\alpha$ is an isomorphism when $r=0$ (that is when $n\leq 5$), and a flopping contraction otherwise.

For $n=2$ both $\Sigma$ and $\Sigma_0$ are singletons. In this setting, most of the statements in this section either remain trivially true, or are clearly false. 
From now on we impose $n\geq 3$.

\subsection{Moduli spaces interpretation}\label{Sigma as mod. sp.}
By definition, $\Sigma=M_\sigma(w)$ for the Mukai vector  $w=v-a=(3,-\cH,n)$. There is a distinguished class in the Néron-Severi group of this moduli space, pulled-back from the base of the contraction.

\begin{remark}
    $\Sigma$ is a fine moduli space if and only if $3$ does not divide $n$.
\end{remark}

The Mukai vector $u=(2,-\cH,2(n-1))$ is orthogonal to $w$, hence the associated Mukai morphism $\theta$ \eqref{mukai morphism} sends $u$ to a class in $\NS(\Sigma)$.
Note that $u$ is a primitive generator of the orthogonal complement of $\Lambda$ in $H^*_{alg}(S,\ZZ)$.
We denote by $\cL$ the class among $\pm \theta(u)$ which lies in the positive cone of $\Sigma$; since $u^2=2$, $\cL$ has Beauville-Bogomolov degree $2$.
The following result is basically a restatement of Lemma \ref{base irred} for $i=0$, but this time we want to highlight the r\^ole of $\cL$.

\begin{prop}\label{cL on Sigma}
    The class $\cL$ is nef and big and lies in the interior of the movable cone of $\Sigma$. Moreover $\alpha=\phi_{k\cL}$ for $k\gg 0$; in particular $\cL$ is ample if and only if $n\leq 5$, while for $n\geq 6$ it spans a flopping wall.
\end{prop}

\begin{proof}
We consider the walls-and-chambers decomposition of $\Stab^\dagger(S)$ associated to $w$.
Let $\cC$ be the chamber in $\Stab^\dagger(S)$ containing $\sigma$ and $\overline{\cC}$ its closure. The map 
$\ell:\overline{\cC}\lra \NS(\Sigma)_\RR$ constructed in \cite{bayermacri1} sends $\overline{\cC}$ to the nef cone and $\cC$ to its interior.
Lemma \ref{base irred} shows that $\sigma_0\in \cC$ for $n\leq 5$ (that is when $r=0$), otherwise $\sigma_0$ only lies in the closure of $\cC$. Equivalently, $\ell(\sigma_0)$ is always nef and big and is ample if and only if $n\leq 5$.
By definition of $\Lambda$, the central charges with respect to $\sigma_0$ of elements of $\Lambda$ align; as one can see from the proof of \cite[Theorem 12.1]{bayermacri}, this is equivalent to 
the condition that $\ell(\sigma_0)\in \theta(v^\perp\cap \Lambda)^\perp$. The latter is generated by $\theta(u)$ and, since $\ell(\sigma_0)$ is nef, $\theta(u)$ is a positive multiple of $\cL$.
By the main result of \cite{bayermacri1}, $\alpha$ is defined by the ray spanned by $\theta(u)$, hence $\alpha=\phi_{k\cL}$ for some $k>0$.
Finally, $\cL$ lies in the interior of the movable cone since $\alpha$ is either an isomorphism or a flopping contraction.
\end{proof}

Since $\cL$ has Beauville-Bogomolov degree $2$, $-R_\cL$ is an integral isometry of $H^2(\Sigma,\ZZ)$ (where again $R_\cL$ denotes the reflection with respect to $\cL$). 

\begin{definition}\label{unique involution}
Let $\iota_\Sigma$ be the unique birational involution of $\Sigma$ acting as $-R_\cL$ in cohomology; its existence follows from the same argument as in \cite[proof of Theorem 1.1, case (i)]{beri}. It is non-symplectic, acts as $-id$ on the transcendental lattice of $\Sigma$, and fixes 
the span of $\cL$ in $\NS(\Sigma)$. 
It is biregular if and only if  $\cL$ is ample, i.e. $n\leq 5$.
\end{definition}

For $n=3$, $\Sigma\simeq\Sigma_0=B$ is a two-polarized K3 surface with $\NS(\Sigma)=\ZZ\cL$.
The covering involution acts as $-R_\cL$ in cohomology. On the other hand, we observed in 
\cite[Section 4.5]{invol1} 
that in that case $\varphi$ induces an involution on $N$ which is biregular, and restricts to $\Sigma\subset N$ as its covering involution: this is \cite[Lemma 35]{invol1} (where the induced involution is denoted $\gamma$). 

It is natural to ask whether this picture extends to higher dimensions. In other words, we can wonder if the involution acting as $-R_\cL$ on the cohomology of $\Sigma$ is induced by $\varphi$. We will answer this question for most  cases in  Subsection \ref{sec: invol}. But we will first observe that, somewhat surprisingly, the group of birational automorphisms $\Bir(\Sigma)$ may be infinite for certain values of $n$; in particular  we cannot in general
show that $\varphi_\Sigma^*=-R_\cL$
simply arguing by unicity, as we do for $S^{[n]}$.

\smallskip
In order to understand whether $\Bir(\Sigma)$ is finite or not, we use Oguiso's techniques relying on the weak version of the movable cone conjecture proved by Markman in \cite{markmansurvey}.
We first need to describe $\NS(\Sigma)$; note that starting from $n=4$ we have $w^2\geq 2$, so $\Sigma$ has Picard rank two by \cite[Theorem 3.6]{bayermacri}.

\begin{lemma}\label{pic sigma}
Let $n\geq 4$. As a lattice, $\NS(\Sigma)=\langle 2 \rangle\oplus \langle-2\frac{t(n-3)}{\gcd(3,n)^2}\rangle$.
\end{lemma}

\begin{proof}
A primitive generator of $w^\perp\cap u^\perp$ in $H^*_{alg}(S,\ZZ)$ is $\frac{1}{\gcd(3,n)}(t,-(2n-3)\cH,t(n-2))$. It corresponds via the Mukai morphism to a generator $\kappa$ of the orthogonal complement of $\cL$ in $\NS(\Sigma)$. The Beauville-Bogomolov degree of $\kappa$ is $-2\frac{t(n-3)}{\gcd(3,n)^2}$. The lattice spanned by $\cL$ and $\kappa$ is a finite index sublattice of $\NS(\Sigma)$. By \cite[Lemma 7.2]{ghs13}, both lattices have the same discriminant, so they coincide by \cite[XIV.0.2, (0.1)]{huybrechtsK3}.
\end{proof}

The positive cone $\rm{Pos}(\Sigma)$ is delimited by the lines spanned by $\frac{\sqrt{t(n-3)}}{\gcd(3,n)}\cL\pm \kappa$, which are rational if and only if $n=7$. For this reason, this case needs to be treated separately.

\begin{prop}\label{card Bir}
For $n=7$, $\Bir(X)$ is finite.
For $n\geq 4$ and $n\neq 7$, $\Bir(X)$ is finite if and only if $\Mov(\Sigma)$ is strictly contained in $\rm{Pos}(\Sigma)$.
\end{prop}

\begin{proof}
Suppose $n\neq 7$.
Since the boundary rays of $\rm{Pos}(\Sigma)$ are irrational, \cite[Theorem 1.3 (2)]{oguiso-pic2} implies that $\Bir(\Sigma)$ is infinite when $\Mov(\Sigma)=\rm{Pos}(\Sigma)$. On the other hand, the chambers in the interior of the positive cone are cut out by the orthogonal complements to the classes of effective divisors with negative Beauville-Bogomolov degree \cite[Section 6.4]{markmansurvey}. So when $\rm{Pos}(\Sigma)\neq \Mov(\Sigma)$, at least one such divisor must exist and we conclude by \cite[Proposition 5.1]{oguiso-pic2}. 

For $n=7$ the boundary rays of $\rm{Pos}(\Sigma)$ are rational, so $\Bir(\Sigma)$ is finite either by \cite[Theorem 1.3 (2)]{oguiso-pic2} or \cite[Proposition 5.1]{oguiso-pic2}.
\end{proof}

\begin{coro}
$|\Bir(\Sigma)|=\infty$ whenever $n=3n'$ for some $n'>1$.
\end{coro}

\begin{proof}
Let prove that $\rm{Pos}(\Sigma)=\Mov(\Sigma)$. If this is not the case, there must exist a class $s$ or $w$ as in \cite[Theorem 12.3]{bayermacri}. The second possibility is easily ruled out,  since $w$ (our $w=(3,-\cH,n)$) has divisibility $\gcd(3,2t,n)=3$. To exclude the first possibility,
we observe that the Mukai morphism $\theta$ sends a spherical class $s\in w^\perp$ to an element of $\NS(\Sigma)$; so by Lemma \ref{pic sigma} $\theta(s)$ must be of the form $X\cL+Y\kappa$ with $X^2-(4n'-1)(n'-1)Y^2=-1$. But a necessary condition for the negative Pell's equation $X^2-\Delta Y^2=-1$ to admit a solution is that no prime number $p$ with $p\equiv 3\pmod{4}$ divides $\Delta$, since $-1$ is not a residue modulo those primes; under our hypothesis there must exist such a prime, since $4n'-1\equiv 3\pmod{4}$.
\end{proof}

As we will see in the next subsection, in most cases the following observation will be sufficient to identify the action of $\varphi_\Sigma$ in cohomology.

\begin{lemma}\label{cohom invol}
For any $n\geq 4$, the only birational automorphism of $\Sigma$ acting trivially on $\NS(\Sigma)$ is the identity. 
\end{lemma}

\begin{proof}
Let $g\in \Bir(\Sigma)$.
Since the transcendental lattice $T$ of $\Sigma$ has odd Picard rank, $g^*$ restricts to $\pm id_T$ on $T$ \cite[Proof of Lemma 4.1]{oguiso}. Hence if $g$ is symplectic, $(g^*)_{|T}=id_T$, and if moreover $g$ acts trivially on $\NS(\Sigma)$, then $g^*=id_{H^2(\Sigma,\ZZ)}$ and we are done.
Suppose instead that $g$ is non-symplectic 
and acts trivially on $\NS(\Sigma)$. It has to be an involution, with invariant lattice $\NS(\Sigma)$ and coinvariant lattice $T$; as a consequence of \cite[Proposition 1.1]{camere}, the discriminant of one of them has to be a power of two. But this is not the case. Indeed, the discriminant of $\NS(\Sigma)$ is $4t(n-3)$. Moreover, since $T$ is isometric to the transcendental lattice of $S$, it has discriminant $2t$ by \cite[XIV.0.2, (0.3)]{huybrechtsK3} since $\NS(S)$ is generated by $\cH$ and $H^2(S,\ZZ)$ is unimodular.
Both discriminants contains $t>1$ as a factor, and $t$ is odd so we are done.
\end{proof}

\subsection{The induced involution}\label{sec: invol}
The birational involution $\varphi$ descends to a birational involution $\overline
{\varphi}$ on $N$. 

\begin{prop}\label{involution Sigma0}
   $\overline{\varphi}$ on $N$ is regular. 
\end{prop}

\proof 
By \cite[Theorem 2.1.27]{lazarsfeld}, there is an ample line bundle $\mathcal{A}$ on $N$ whose pullback via $c$ is a positive multiple of $\cH_n-2\delta$. In particular, $\mathcal{A}$ is invariant under the action of $\overline{\varphi}$.
Choose an integer $m$ such that $m\mathcal{A}$ is very ample. For any section 
$\sigma$ of $m\mathcal{A}$,  $\sigma \circ\overline{\varphi}$ is a rational section of $\overline{\varphi}^*(m\mathcal{A})=m\mathcal{A}$, and since $N$ is normal by construction (see \cite[Theorem 1.4(a)]{bayermacri1}), 
it extends to a regular section. We thus get a linear involution $\overline{\varphi}^*$
of $H^0(N,m\mathcal{A})$. But since $m\mathcal{A}$ is very ample, $N$ embeds inside $\PP(H^0(N,mA)^\vee)$ and  by construction, the induced
projective automorphism of $\PP(H^0(N,mA)^\vee)$ restricts to $\overline{\varphi}$. In particular, $\overline{\varphi}$ is regular. \qed 

\medskip

Since the base of the contraction is the singular locus of $N$, $\overline{\varphi}$ restricts to a biregular involution on it. Via the normalization $\eta$, a biregular involution on $\Sigma_0$ is also induced, and via $\alpha$ a birational involution $\varphi_\Sigma$ on $\Sigma$. 

\begin{coro}
$\varphi$ induces a birational involution $\varphi_\Sigma$ of 
$\Sigma$.    
\end{coro}

Contrary to the previous ones, this involution $\varphi_\Sigma$ is a priori not regular, and we will show later on that it is indeed \textit{not} biregular, except for $n=3,4,5$ for which $\Sigma\simeq \Sigma_0$.

We summarize the situation in a diagram:

\[
\xymatrix{
  \Sigma\ar[d]_\alpha\ar@{-->}@(u,l)_{\varphi_\Sigma} &  I \ar@{^{(}->}[r]\ar[d] &  S^{[n]}\ar@{-->}@(u,r)^{\varphi} \ar[d]_c\\
\Sigma_0 \ar@(d,l)\ar[r]^\eta_{1:1} &  B\ar@{^{(}->}[r]  &  N\ar@(d,r)_{\overline{\varphi}}.
}
\]

\medskip
In order to understand $\varphi_\Sigma$, we will relate it to the main result of \cite{faenzi-menet}.  For a sheaf $\cE$ from $\Sigma=M_\sigma(w)$, a straightforward computation shows that the Mukai vector of $\cE(\cH)$ is $(3, 2\cH, 5n-3)$. So  \cite[Theorem 4.4]{faenzi-menet} applies with $k=1$ (and $g=4n-2$); with their notations, $\Phi=\Phi_{\cU_2^\vee, 1}$ is a birational involution of $M_\sigma(w)$. By definition, up to a twist by $\cH$, it sends a generic sheaf $\cE\in M_\sigma(w)$ to a sheaf $\cF\in M_\sigma(w)$ that fits into exact sequences
\begin{align}\label{involution-for bundles} 
0\lra \cE^\vee\lra Hom(\cU_2,\cE)^\vee\otimes \cU_2^\vee\lra \cF(\cH)\lra {\mathcal Ext}^1(\cE,\cO_S)\lra 0, \\
0\lra \cF^\vee\lra Hom(\cU_2, \cF)^\vee\otimes \cU_2^\vee\lra \cE(\cH)\lra {\mathcal Ext}^1(\cF,\cO_S)\lra 0,
\end{align}
where the sheaves ${\mathcal Ext}^1(\cE,\cO_S)$ and ${\mathcal Ext}^1(\cF,\cO_S)$ have finite support
\cite[Lemma 3.3]{faenzi-menet}. 
These sequences are dual one to the other, and induce an isomorphism
$$ Hom(\cU_2, \cE)^\vee\lra Hom(\cU_2, \cF).$$

\begin{prop}\label{descr varphiSigma}
$\varphi_\Sigma=\Phi$.
\end{prop}

\proof
By Lemma \ref{aumoins3 debut},
the open subset $U^{(0)}\subset \Sigma$ parametrizes 
sheaves $\cE$ such that  $h^0(\cU_2\otimes\cE)=3$. Via $\alpha$, we see $U^{(0)}$ as an open subset of $\Sigma_0$. By \eqref{1ere strate}, the stratum $I_0\subset I$ is a $\PP^2$-fibration over $U^{(0)}$, hence a smooth codimension two locally closed subvariety of $S^{[n]}$. The symplectic form on the Hilbert scheme allows to identify its normal bundle to the relative cotangent bundle. If we blow-up $I_0$, we thus get an exceptional bundle with fibers $\PP (Hom(\cU_2,\cE)/\cO(-1))$ over $I_0$; note that since $Hom(\cU_2,\cE)$ is three-dimensional, this is a relative flag bundle over $U^{(0)}$. So a point in the exceptional divisor should be interpreted as a triple $(\cE,[u],[v])$
for $\cE\in U^{(0)}$, $u\in Hom(\cU_2,\cE)$, $v\in Hom(\cU_2,\cE)^\vee$, with $u$ and $v$ non-zero and orthogonal.  

If $\cF=\Phi(\cE)$, then $v$ defines an element of $Hom(\cF,\cU_2^\vee)$, from which we get an exact sequence 
\begin{equation}\label{extension'}
0\lra \cU_2\lra \cF \lra \cI_{Z'}\lra 0
\end{equation}
for some scheme $Z'\in I$. We claim that the association $Z\mapsto Z'$ extends $\varphi$, which will prove the Proposition. By construction of $\varphi$ (Proposition \ref{intro-involution}), we just need to show that there exists a global section of $\cU_2^\vee$ vanishing both on $Z$ and $Z'$.  

In order to construct such a section, we start by dualizing the sequence 
(\ref{extension'}), which yields  
\begin{equation}\label{dual-extension} 
0\lra \cO_S\lra \mathcal{F}^\vee\stackrel{v}{\lra} \cU_2^\vee\lra \cO_{Z'}\lra 0.
\end{equation}
In particular there is an induced section $\sigma_v$ of $\cF^\vee$ (at least up to non-zero scalar). We can easily give an explicit expression of this section, say at a point $s\in S$ at which $\cF^\vee$ is locally free. Let $f_1,f_2,f_3$ be a basis of the fiber $\cF^\vee_s$;
since $\det(\cF^\vee)=H$, we see $f_1\wedge f_2\wedge f_3$ as a generator at 
$s$ of $H=\det(\cU_2^\vee)$. Therefore the following expression makes sense 
$$\sigma_v(s)=\frac{v(f_1)\wedge v(f_2)}{f_1\wedge f_2\wedge f_3}f_3+
\frac{v(f_2)\wedge v(f_3)}{f_1\wedge f_2\wedge f_3}f_1+\frac{v(f_3)\wedge v(f_1)}{f_1\wedge f_2\wedge f_3}f_2 \in\cF^\vee_s,$$
and does not depend on any choice. Then we can use any $v'\in Hom(\cU_2,\cF)$ to deduce a section $\sigma_{v,v'}$ of $\cU_2^\vee$ outside the singular set of $\cF$; since this singular set is finite, $\sigma_{v,v'}$ extends to a global section.

\begin{lemma} Suppose $Z, Z'$ are reduced. Then:
\begin{enumerate}
    \item $\sigma_v$ vanishes along $Z'$,
    \item $\sigma_{v,v}=0$,
    \item if $v'$ is orthogonal to $u$, then $\sigma_{v,v'}$ vanishes along $Z$.
\end{enumerate}
\end{lemma}

\noindent {\it Proof of the Lemma.} The exact sequence 
(\ref{dual-extension}) shows that on $Z$ the rank of $v$ is only one, so the wedge products in the definition of $\sigma_v$ all vanish, implying (1). 
The expression of $\sigma_{v,v}$ is skew-symmetric in the three vectors $v(f_1), v(f_2), v(f_3)$, that  belong to a vector space of dimension two; so it has to vanish, proving (2). 

In order to prove the last assertion, we use the exact sequence 
$$0\lra \cE^\vee\lra Hom(\cF, \cU_2)\otimes \cU_2^\vee\lra 
\cF(\cH)\lra 0,$$
which holds outside the singular locus of $\cE$.
If we use a local basis $x_1,x_2$ of $\cU_2^\vee$ at $s$, this means that $\cE$ can be described locally as the bundle of combinations $w_1\otimes x_1+w_2\otimes x_2$, where $w_1, w_2\in Hom(\cF, \cU_2)$, such that 
\begin{equation}\label{relation}
\langle w_1(f), x_1\rangle + \langle w_2(f), x_2\rangle =0
\end{equation}
for any local section $f$ of $\cF$. The morphism from $\cE^\vee$ to $\cU_2^\vee$ defined by 
$u$, or equivalently by the orthogonal vectors $v$ and $v'$, simply sends 
$$w_1\otimes x_1+w_2\otimes x_2\mapsto (w_1\wedge v\wedge v')x_1+(w_2\wedge v\wedge v')x_2,$$
seen as a section of $\cU_2^\vee$ once we have 
trivialized  $ \wedge^3Hom(\cF, \cU_2)\simeq \CC$. Over $Z$, this morphism has rank one, say its image is generated by $x_1$. So 
$w_2$ must belong to $\langle v,v'\rangle$; if we complete $v,v'$ with $v''$ to get a basis of $Hom(\cF, \cU_2)$ and write $w_i=z_i v+z'_iv'+z''_i v''$, this means that the equation $z''_2$ must be one of the equations given by (\ref{relation}). Hence there must exist an $f$ such that $v(f), v'(f), v''(f)$ are parallel to $x_2$, while  $v(f), v'(f)$ are parallel to $x_1$. 
Then necessarily $v(f)=v'(f)=0$, and choosing $f_3=f$ in the expression 
$$\sigma_{v,v'}(s)=\frac{v(f_1)\wedge v(f_2)}{f_1\wedge f_2\wedge f_3}v'(f_3)+
\frac{v(f_2)\wedge v(f_3)}{f_1\wedge f_2\wedge f_3}v'(f_1)+\frac{v(f_3)\wedge v(f_1)}{f_1\wedge f_2\wedge f_3}v'(f_2),$$
we can clearly conclude that it vanishes.
\qed 

\medskip \noindent {\it Conclusion.}
Up to scalar, the Lemma implies that   $\sigma_{v,v'}$ only depends on the flag 
$\langle v\rangle\subset \langle v,v'\rangle$. If we choose, $\langle v,v'\rangle = u^\perp$, we thus get a uniquely defined section which vanishes along $Z\cup Z'$. This concludes the proof. \qed 

\medskip

Since $\varphi_\Sigma=\Phi$ is a non-symplectic involution \cite[Proposition 5.11]{faenzi-menet}, $\varphi_\Sigma$ and $\overline{\varphi}_{|B}$ are not the identity! 
Recall that $-R_\cL$, minus the reflection with respect to $\cL$, acts in $H^2(\Sigma,\ZZ)$ as the identity on the span of $\cL$ and as $-id$ on the orthogonal complement.

\begin{coro}\label{action in cohom}
For $n\geq 6$, $\varphi_\Sigma^*=-R_\cL$. In particular $\varphi_\Sigma$ is birational, but not biregular on any hyperK\"ahler birational model of $\Sigma$.
\end{coro}

\begin{proof}
We know that $\overline{\varphi}$ fixes an ample class $\cA$ on $B$ by Proposition \ref{involution Sigma0}. Denote also by $\cA$ its pullback to the normalization $\eta:\Sigma_0\lra B$.
Since $\varphi_\Sigma$ is induced by $\alpha:\Sigma\lra \Sigma_0$, the pullback of $\cA$ by $\alpha$ is fixed by $\varphi_\Sigma$. Since $n\geq 6$, by Proposition \ref{cL on Sigma} $\alpha$ is a flopping contraction, hence $\alpha^*\cA$ is a nef and big class in the interior of the movable cone which is not ample. All such classes are positive multiples of $\cL$, so $\varphi_\Sigma$ fixes them. By the proof of Lemma \ref{cohom invol}, since $\varphi_\Sigma$ is not the identity, it fixes exactly one class in $\NS(\Sigma)$ (up to constant), and then it has to act as $-id$ on $\cL^\perp\cap \NS(\Sigma)=\ZZ\kappa$. We conclude that the action of $\varphi_\Sigma$  on $\NS(\Sigma)$ coincides with that of the automorphism $\iota_\Sigma$ from Definition \ref{unique involution}. Since both are involutions, $(\varphi_\Sigma\circ \iota_\Sigma)^*_{|\NS(\Sigma)}=id_{\NS(\Sigma)}$, hence 
$\varphi_\Sigma\circ \iota_\Sigma$ is the identity by Lemma \ref{cohom invol} and $\varphi_\Sigma=\iota_\Sigma^{-1}=\iota_\Sigma$.

For $\varphi_\Sigma$ to be biregular on $\Sigma'$ a hyperK\"ahler birational model of $\Sigma$, the action of $\varphi_\Sigma$ in cohomology must fix a class in the interior of the chamber of the movable cone corresponding to $\Sigma'$. But all the movable classes fixed by $\varphi_\Sigma^*$ lie in the span of $\cL$, which generates a wall.
\end{proof}

The case $n=3$ is settled in \cite{invol1}, so the remaining cases are $n=4,5$, for which $\varphi_\Sigma$ is a biregular involution.
For $n=4$, we can fully recover the result above, by a study of the group of birational automorphisms.

\begin{prop}\label{epw}
For $n=4$, $\varphi_\Sigma^*=-R_\cL$. Moreover $(\Sigma,\cL)$ is a double EPW sextic, whose covering involution is $\varphi_\Sigma$. 
\end{prop}

\begin{proof}
For $n=4$ we have $t=13$. Consider the basis $\lbrace \cL,\kappa \rbrace$ of $\NS(\Sigma)$ described in Lemma \ref{pic sigma}. For $(a,b)$ any solution of the negative Pell's equation $X^2-tY^2=-1$, for example $(a,b)=(18,5)$, the orthogonal complement of $a\cL+b\kappa$ cuts a wall in the interior of the positive cone of $\Sigma$ by \cite[Theorem 12.3]{bayermacri}.
Thus $\Bir(\Sigma)$ is finite by Proposition \ref{card Bir}.

Recall the involution $\iota_\Sigma$ from Definition \ref{unique involution}. 
By Lemma \ref{cohom invol}, the invariant lattice of the involution $\varphi_\Sigma$ in $\NS(\Sigma)$ has rank one, so it is generated by a primitive ample class $\cL'$.
Suppose $\cL\neq \cL'$, so that $\lbrace \cL,\cL' \rbrace$ is a $\QQ$-basis for $\NS(\Sigma)$. We denote by $\lambda$ the Beauville-Bogomolov product of $\cL$ and $\cL'$; both $\cL$ and $\cL'$ have Beauville-Bogomolov degree two by \cite[Proposition 3.1]{camere}, and we need a $\lambda>2$ for the lattice to be hyperbolic. The induced action of $\varphi_\Sigma\circ \iota_\Sigma$ in the basis $\lbrace \cL,\cL' \rbrace$ is given by
$$
\begin{bmatrix}
-1-\lambda & -\lambda \\
\lambda & \lambda^2-1 
\end{bmatrix},
$$
with  characteristic polynomial $X^2+(2-\lambda^2)X+1$. This polynomial is never cyclotomic, so 
$\varphi_\Sigma\circ \iota_\Sigma$ has infinite order. This is absurd since $|\Bir(\Sigma)|<\infty$. So $\cL=\cL'$ and $\varphi_\Sigma=\iota_\Sigma$ by Lemma \ref{cohom invol}.

We know that $(\Sigma,\cL)$ is a polarized pair, with $\cL$ of Beauville-Bogomolov degree $2$. The moduli space of $2$-polarized hyperK\"ahler manifolds of $K3^{[2]}$-type is irreducible \cite{apostolov}; 
since $\Sigma$ has Picard rank two, with transcendental lattice of discriminant $2t=26$, the period points of our pairs $(\Sigma,\cL)$'s 
are dense in an irreducible component of the divisor denoted by $\cD^{(1)}_{2,26}$ in \cite{debarremacri}.
As observed in \cite[Example 6.3]{debarremacri}, from \cite[Theorem 8.1]{dim} one can deduce that $(\Sigma,\cL)$ is a double EPW sextic. The covering involution of a double EPW sextic always acts in cohomology as minus the reflection with respect to the polarization, so it must coincide with $\varphi_\Sigma$.
\end{proof}

\begin{remark}
The involution $\overline{\varphi}$ on $N$ induces also an involution on $\Sigma^{(k)}_0$, hence on $\Sigma^{(k)}$, for any $k\geq 0$. It would be interesting to adapt the statements in this section to those situations. This is all the more tempting that if $\cE$ has Mukai vector 
$v^{(k-1)}$, then the Mukai vector of the twist $\cE(\cH)$ is exactly as in \cite[Theorem 4.4]{faenzi-menet}; it is very  likely that the two birational involutions coincide.
\end{remark}

\subsection{The Pl\"ucker variety}

We call \textit{Pl\"ucker space} the locus in $|\cH_n-2\delta|^\vee$ where the projection  $p_{V^\vee_{2n+1}}$ is not defined (see \eqref{subsystem map}), i.e. the projectivization of
\begin{equation}\label{generic}
V_{2n+1}^\perp =\left\{ H\in |I_n(Sec^{n-2}(S))|^\vee \mbox{ such that }\mbox{ }H\supset V_{2n+1}\right\}.
\end{equation}
Its dimension is $\frac{(n+2)(n+1)}{2}-(2n+1)=\frac{n(n-1)}{2}$.
The following statement follows from the definition of $V^\perp_{2n+1}$ and the fact that  $I$ is the base locus of the linear subsystem $|V_{2n+1}|$, cf. subsection \ref{sec: pfaff}.

\begin{lemma}
$\phi_{\cH_n-2\delta}$ sends $I$ into the the Pl\"ucker space $|V^\perp_{2n+1}|$.
\end{lemma}

We define the  \textit{Pl\"ucker variety} to be 
$$\overline{\Sigma}:=\phi_{\cH_n-2\delta}(I)=|V^\perp_{2n+1}|\cap 
\phi_{\cH_n-2\delta}(S^{[n]}).$$
It plays the same role as the Pl\"ucker point for $n=2$; 
for $n=3$ it is equal to the Pl\"ucker space $|V^\perp_7|\simeq \PP^2$; starting from $n\geq 4$ the two are different.

\medskip
Recall that by Corollary \ref{c stein}, the flopping contraction $c$ is the first factor of the Stein factorization of $\phi_{\cH_n-2\delta}$. We denote the second factor by $\nu$.
Via $\nu$, the base of the contraction $B$ is sent to $\overline{\Sigma}\subset |V^\perp_{2n+1}|$. 
We denote by $\tilde{\alpha}$ the composition of $\alpha$ with the normalization map $\eta$, 
and by $f$ the composition of $\tilde{\alpha}$ with $\nu$, as in the following diagram:
\[
\xymatrix{
  \Sigma\ar@{-->}@(u,l)_{\varphi_\Sigma}\ar[dr]_{\tilde{\alpha}}\ar[dd]_f &  I \ar@{^{(}->}[r]\ar[d] &  S^{[n]}\ar@{-->}@(u,r)^{\varphi} \ar[d]_c \ar\ar@/^1.0pc/@{-->}[ddr]^{\phi_{V_{2n+1}}} &\\
 &  B\ar@{^{(}->}[r]\ar[d]\ar[dl]  &  N \ar[d]_\nu &\\
\overline{\Sigma}\ar@{^{(}->}[r] & |V_{2n+1}^\perp| \ar@{^{(}->}[r] & |\cH_n-2\delta|^\vee \ar@{-->}[r]_{p_{V^\vee_{2n+1}}} & |V_{2n+1}|^\vee
}
\]
By construction, the image of $\Sigma$ via $f$ is $\overline{\Sigma}$.

\begin{prop}
    Suppose $n\ge 6$. The Pl\"ucker variety $\overline{\Sigma}$ is the image of $I$ by the full linear system
    $|\cO_I(\cH_n-2\delta)|$.
\end{prop}

\proof 
By Theorem \ref{thm I=J}, $I=J$ is a degeneracy locus of the expected dimension, so its  ideal sheaf can be resolved by the 
Eagon-Northcott complex 
\begin{equation}\label{eagon-northcott1}
0\lra ((\cU_2^\vee)^{[n]})^\vee\lra \mathcal{O}_{S^{[n]}}^{\oplus (2n+1)}\lra 
\mathcal{I}_I(\det ((\cU_2^\vee)^{[n]}))\lra 0
\end{equation}
(see e.g. \cite{weyman}). Moreover $\det ((\cU_2^\vee)^{[n]})=\cH_n-2\delta $ (see \cite[Lemma 1.5]{wandel}).

\medskip
We will prove that  $H^0(S^{[n]},((\cU_2^\vee)^{[n]})^\vee)$ and $H^1(S^{[n]},((\cU_2^\vee)^{[n]})^\vee)$ vanish.
By Serre duality, this follows from: 

 \begin{lemma}\label{vanishing1}
 Let $\cE$ be a vector bundle on $S$ with no higher cohomology. Then
   $h^q(S^{[m]},\cE^{[m]})=0$ for  $q\ge 2m-1$,  and 
  $h^{2m-2}(S^{[m]},\cE^{[m]})\le m h^0(S,\cE).$
\end{lemma}

\proof This can be checked by induction on $m$, following the strategy of 
\cite[section 6]{danila}. This involves the nested Hilbert scheme $S^{[m-1,m]}$ and its projection $p_m$ to $S^{[m]}$ and $\phi$ to $S^{[m-1]}\times S$. The first observation 
\cite[Assertion 5) page 21]{danila} is that $ H^q(S^{[m]},\cE^{[m]})$ is a factor of 
 $H^q(S^{[m-1,m]},p_m^*\cE^{[m]})$. In order to control the latter cohomology group, 
 an exact sequence is constructed \cite[Assertion 2) page 21 for $A=0$]{danila}:
 $$0\lra p_m^*\cE^{[m]}\lra \phi^*(\cE^{[m-1]}\boxtimes \cO_S)\oplus \phi^*(\cO_{S^{[m-1]}}\boxtimes \cE) \lra \phi^*(\cO_{S^{[m-1]}}\boxtimes \cE)_{|\Delta}\lra 0,$$
where $\Delta$  denotes the exceptional divisor. Moreover, the cohomology groups of the 
terms of this sequence can be controlled through the following isomorphisms \cite[Assertions 4) and 6) page 21]{danila}:
$$H^q(S^{[m-1,m]},\phi^*(\cE^{[m-1]}\boxtimes \cO_S))=H^q(S^{[m-1]},\cE^{[m-1]})\oplus 
H^{q-2}(S^{[m-1]},\cE^{[m-1]}),$$
$$H^q(S^{[m-1,m]},\phi^*(\cO_{S^{[m-1]}}\boxtimes \cE))=H^q(S^{[m-1]}\times S,\cO_{S^{[m-1]}}\boxtimes \cE),$$
$$H^{q-1}(S^{[m-1,m]}, \phi^*(\cO_{S^{[m-1]}}\boxtimes \cE)_{|\Delta})=
H^{q-1}(S^{[m-1]},\cE^{[m-1]}).$$
To prove that $H^q(S^{[m]},\cE^{[m]})=0$ it suffices to check that these three terms vanish. 
For $q\ge 2m-1$ we can use induction for the first one. Since the bundle $\cE$ has no higher cohomology, by the Kunneth formula the second one vanishes for degree reasons. The third one also vanishes by induction. 

For $q=2m-2$, the first term reduces to 
$H^{2m-4}(S^{[m-1]},\cE^{[m-1]})$, which we can control by induction. The last term
vanishes by what we have just proved. The second one does not vanish since it reduces to  $H^0(S,\cE)$, so induction yields the asserted bound. \qed

\medskip
As a consequence of the Lemma, 
$h^0(\mathcal{I}_I(\cH_n-2\delta))=2n+1$, and the restriction morphism 
    $$H^0(S^{[n]},\cH_n-2\delta)\lra H^0(I,\cO_I(\cH_n-2\delta))$$
    has an image of dimension $\frac{n(n-1)}{2}$, equal to the dimension of $V_{2n+1}^\perp$. 
At this point, is not clear that this image is the full linear series, but by the previous estimate 
$h^1(\cI_I(\cH_n-2\delta))$ is at most $n(2n+1)$, hence  $h^0(\mathcal{O}_I(\cH_n-2\delta))$
is at most $\frac{n(n-1)}{2}+n(2n+1)$.
\medskip 

On the other hand, since $I$ contracts to $B$ with connected fibers, $H^0(\mathcal{O}_I(\cH_n-2\delta))$ 
is isomorphic to $H^0(B,\mathcal{A})$, and to $H^0(\Sigma_0,\mathcal{A})$ (we use the same notations as Corollary \ref{action in cohom}) and $H^0(\Sigma,\alpha^*\mathcal{A})$ as well since $\eta$ and $\alpha$ also have connected fibers. As we already saw in Corollary \ref{action in cohom},  this is 
$H^0(\Sigma,k\cL)$ for some $k>0$ \cite[Theorem 2.1.27]{lazarsfeld}. We denote by $q(\cL)$ the Beauville-Bogomolov degree of $\cL$; since $\Sigma$ is of $K3^{[n-2]}$-type, we know 
(see \cite[3.3]{debarre}) that 
$$h^0(\Sigma,k\cL)=\binom{\frac{k^2}{2}q(\cL)+n-1}{n-2},$$
and moreover $q(\cL)=2$. For $n\ge 6$ we deduce that $k=1$, and $H^0(\mathcal{O}_I(\cH_
n-2\delta))$ is isomorphic to $V_{2n+1}^\perp$. \qed

\begin{coro}
    Suppose $n\geq 6$. The image of $\Sigma$ via $f=\phi_\cL$ is the Pl\"ucker variety $\overline{\Sigma}$.
\end{coro}

\begin{remark}
For $n=2$, the K3 surface $S\subset \PP^{t+1}=\PP^6$ is an instance of a Gushel-Mukai variety, a class of complete intersections inside the Grassmannian $G(2,V_5)$, which have been extensively studied by many authors. To any Gushel-Mukai variety one can associate a so-called Pl\"ucker point, that plays a fundamental r\^ole in their study, see \cite{ogrady1, iliev, debkuz-classbir}.
Our Pl\"ucker varieties are designed to be generalizations of the Pl\"ucker points for $n\geq 3$, and to help understanding 
to what extent the Gushel-Mukai picture generalizes.
\end{remark}

\smallskip\noindent {\bf Conjecture.}
$\phi_\cL :\Sigma\to |V^\perp_{2n+1}|\simeq \PP^{\frac{(n+1)(n-2)}{2}}$  
is generically finite of degree two over the  Pl\"ucker variety.

\medskip For $n=4$ we already know by Proposition  \ref{epw} that $\Sigma$ is a double EPW sextic, and we expect $\overline{\Sigma}\subset \PP^5$ to be an EPW sextic. The next case $n=5$ would be that of K3 surfaces of genus $18$, for which a model was found by Mukai: the general such surface is the common zero locus, inside the orthogonal Grassmannian $OG(3,9)$, of five sections of the rank two spinor bundle \cite{mukai-g=18}. In this case we get a sixfold 
$\overline{\Sigma}\subset \PP^9$, presumably of trivial canonical class, in particular subcanonical.
Can we describe it in terms of Lagrangian bundles?

\appendix

\section{Computing the walls of the movable cone}\label{appendix}

We use a computer program to show that the movable cone of $S^{[n]}$ has exactly two chambers for $n\leq 200$, that is, with the notation of Proposition \ref{nC=1}, $C_n=1$ for $n\leq 200$. 

We explain how the code attached below works.
As one can see in the proof of \cite[Theorem 6.4]{cattaneo}, walls in the interior of the movable cone are spanned by classes $X\cH_n-2tY\delta$ with $(X,Y)$ as in \eqref{eq alpha et rho} lying in the interior of the movable cone.

Since $t$ is $n$-irregular \cite[Lemma 3.6]{beri}, the middle wall -- the $C_n$-th one -- is spanned by $\cH_n-2\delta$, so to prove $C_n=1$ it is sufficient to show that no $X\cH_n-2tY\delta$ lies in the cone $\Pi$ spanned by $\cH_n-2\delta$ and $\cH_n$, or equivalently no solution $(X,Y)$ satisfies $\frac{Y}{X}<\frac{1}{t}$.
Let  $$A(\alpha,\rho)=\alpha^2-4\rho(n-1).$$ 
Since $\frac{Y}{X}=\sqrt{\frac{X^2-A(\alpha,\rho)}{4t(n-1)}}$, the condition can be rewritten as 
\begin{equation}\label{bound x}
0<X<tA(\alpha,\rho).
\end{equation}
Since $A(\alpha,\rho)>0$, we always have $0<Y<X$.

\medskip

The three cases in \eqref{eq alpha et rho} are listed in the program as Cases A, B and C. Note that $(\rho,\alpha)=(-1,1)$ is avoided, since it always admits a solution $(X,Y)=(t,1)$ corresponding to the middle wall and further solutions do not produce classes cutting walls in $\Pi$ \cite[Remark 2.2]{cattaneo}: this is \cite[Lemma 3.6]{beri}, see also Subsection \ref{sec: small c}.

The function 
\texttt{search\_solutions} takes $\alpha$, $\rho$ and $n$ as inputs. For any integer $X$ in the interval \eqref{bound x}, and any integer $Y$ such that $0<Y<X$, it checks that, whenever $X\equiv \pm \alpha \pmod{2(n-1)}$, the pair $(X,Y)$ is not a solution of \eqref{eq alpha et rho}.

\medskip

\begin{verbatim}
import math

i = 0
def search_solutions(alpha, rho, n):
    global i
    i += 1
    sqrt_argument = (4 * n - 3) * (alpha * alpha - 4 * rho * (n - 1))
    if sqrt_argument < 0:
        return
    range_x = range(1, int(math.sqrt(sqrt_argument))+1)
    range_y = range(1, int(math.sqrt(sqrt_argument)))
    for x in range_x:
        alpha_mod = alpha % (2 * (n - 1))
        x_mod = x % (2 * (n - 1))
        if not ((alpha_mod == x_mod or alpha_mod == -x_mod) 
        or alpha_mod == (2 * (n - 1)) -x_mod):
           continue
        for y in range_y:
            lh_exp = x * x - 4 * (4 * n - 3) * (n - 1) * y * y
            rh_exp = alpha * alpha - 4 * rho * (n - 1)
            if rh_exp == lh_exp:
                print(f"n = {n},\t alpha = {alpha},\t rho = {rho}")
                print(f"Solution: x = {x}, y = {y}")


for N in range(4, 201):
    # Case A
    for alpha in range(2, N):
        search_solutions(alpha=alpha, rho=-1, n=N)
    # Case B
    for alpha in range(3, N):
        search_solutions(alpha=alpha, rho=0, n=N)
    # Case C
    range_rho = range(1, int((N - 1) / 4))
    for rho in range_rho:
        range_alpha = range(4 * rho + 1, N)
        for alpha in range_alpha:
            search_solutions(alpha=alpha, rho=rho, n=N)
print("This is the end")
\end{verbatim}

    \bibliographystyle{amsplain}

\bibliography{main.bib}

\end{document}